\documentclass[a4paper]{article}
\usepackage{graphicx}
\usepackage{subcaption}

\usepackage[utf8]{inputenc}
\usepackage{amsfonts}
\usepackage{amsmath}
\usepackage{amsthm}
\usepackage{amssymb}
\usepackage{verbatim}
\usepackage{setspace}
\usepackage{tikz}
\usepackage{slashed}
\usepackage[margin=1in]{geometry}
\usepackage{mathtools}
\usepackage{url}
\usepackage{graphicx}
\usepackage{enumitem}
\usepackage{tikz-cd}
\usepackage[hidelinks]{hyperref}
\usepackage{mathrsfs}
\usepackage{caption}
\usepackage{subcaption}
\usepackage{mathtools}
\usepackage[normalem]{ulem}

\usetikzlibrary{matrix}
\usetikzlibrary{arrows}

\theoremstyle{plain}
\newtheorem{thm}{Theorem}[section]
\newtheorem{lemma}[thm]{Lemma}
\newtheorem{cor}[thm]{Corollary}
\newtheorem{prop}[thm]{Proposition}
\newtheorem{propdef}[thm]{Proposition/Definition}

\theoremstyle{definition}
\newtheorem{defn}[thm]{Definition}

\theoremstyle{remark}
\newtheorem{ex}[thm]{Example}

\makeatletter
\DeclareRobustCommand{\cev}[1]{%
  {\mathpalette\do@cev{#1}}%
}
\newcommand{\do@cev}[2]{%
  \vbox{\offinterlineskip
    \sbox\z@{$\m@th#1 x$}%
    \ialign{##\cr
      \hidewidth\reflectbox{$\m@th#1\vec{}\mkern4mu$}\hidewidth\cr
      \noalign{\kern-\ht\z@}
      $\m@th#1#2$\cr
    }%
  }%
}
\makeatother

\def\GL{\operatorname{GL}}
\def\O{\operatorname{O}}
\def\SO{\operatorname{SO}}
\def\dom{\operatorname{dom}}
\def\t{\operatorname{t}}

\def\supp{\operatorname{supp}}
\def\id{\operatorname{id}}
\def\hom{\operatorname{Hom}}
\def\germ{\operatorname{germ}}
\def\Diff{\operatorname{Diff}}
\def\Tr{\operatorname{Tr}}
\def\Fr{\operatorname{Fr}}
\def\ev{\operatorname{ev}}
\renewcommand{\d}{\ensuremath{d}}
\newcommand{\w}{{\mathchoice{\,{\scriptstyle\wedge}\,}{{\scriptstyle\wedge}}
{{\scriptscriptstyle\wedge}}{{\scriptscriptstyle\wedge}}}}
\newcommand{\composed}{\circ}
\newcommand{\cinfinity}[1]{\mathcal{C}^\infty\left(#1\right)}
\DeclareMathOperator{\so}{\mathfrak{so}}

\newcommand*{\tBun}{\mathrm{T}}
\newcommand*{\tFF}{\tBun\!\FF}
\newcommand*{\nuFF}{\nu\!\FF}

\DeclareSymbolFont{script}{U}{eus}{m}{n}
\DeclareMathSymbol{\Wedge}{0}{script}{"5E}

\DeclareMathOperator{\jetexp}{e}

\def\FF{\operatorname{\mathcal{F}}}

\def\CC{\operatorname{\mathcal{C}}}

\def\CC{\operatorname{\mathcal{C}}}

\def\UU{\operatorname{\mathcal{U}}}
\def\VV{\operatorname{\mathcal{V}}}
\def\TT{\operatorname{\mathcal{T}}}
\def\HH{\operatorname{\mathcal{H}}}

\def\NN{\operatorname{\mathcal{N}}}

\def\XX{\operatorname{\mathcal{X}}}

\DeclareMathOperator{\I}{\mathcal{I}}

\def\FS{\operatorname{\mathscr{F}}}

\def\US{\operatorname{\mathscr{U}}}

\def\RB{\operatorname{\mathbb{R}}}

\def\NB{\operatorname{\mathbb{N}}}

\def\AF{\operatorname{\mathfrak{A}}}

\def\XF{\operatorname{\mathfrak{X}}}
\def\gf{\operatorname{\mathfrak{g}}}

\def\af{\operatorname{\mathfrak{a}}}
\def\kf{\operatorname{\mathfrak{k}}}
\def\glf{\operatorname{\mathfrak{gl}}}

\begin{document}

\title{Chern-Weil theory for Haefliger-singular foliations}
\author{Lachlan E. MacDonald\\Australian Institute for Machine Learning\\The University of Adelaide\\Adelaide, SA, 5000 \and Benjamin McMillan\\School of Mathematical Sciences\\
  The University of Adelaide\\
  Adelaide, SA, 5000}

\date{March 2025}

\maketitle

\begin{abstract}
  We give a Chern-Weil map for the Gel'fand-Fuks characteristic classes of Haefliger-singular foliations, those foliations defined by smooth Haefliger structures with dense regular set. Our characteristic map constructs, out of singular geometric structures adapted to singularities, explicit forms representing characteristic classes in de Rham cohomology. The forms are functorial under foliation morphisms. We prove that the theory applies, up to homotopy, to general smooth Haefliger structures: subject only to obvious necessary dimension constraints, every smooth Haefliger structure is homotopic to a Haefliger-singular foliation, and any morphism of Haefliger structures is homotopic to a morphism of Haefliger-singular foliations. As an application, we provide a generalisation to the singular setting of the classical construction of forms representing the Godbillon-Vey invariant.
\end{abstract}

\section{Introduction}\label{s1}
In this paper, we give a Chern-Weil construction of the Gel'fand-Fuks characteristic classes of certain singular foliations in terms of singular metrics and connections.  Our methods apply specifically to singular foliations arising from Haefliger structures \cite{ha3}, which are the closest singular cousins of regular foliations, as their leaves are still locally determined by complete families of first integrals.  We show that this class is sufficiently general for our theory to apply (up to homotopy) to all smooth Haefliger structures on sufficiently high-dimensional manifolds.  Our work opens the way for further study of the topology of Haefliger's classifying space with all the flexibility afforded by singularities (cf. \cite{thurston1}), and opens up the potential advancement of the study of singular foliations via noncommutative geometry and index theory, which has made great strides in recent years \cite{iakovos2, iakovos7,iakovos3, iakovos}.

\noindent\textbf{The characteristic classes of regular foliations:}
A \emph{regular foliation of codimension $q$} on an  $n$-manifold $M$ is given by an involutive subbundle $\tFF \subset \tBun M$ of rank $n-q$.
By the Frobenius theorem, the involutivity of \( \tFF \) is equivalent to a decomposition \( \FF \) of $M$ as a union of non-intersecting, immersed submanifolds of dimension $n-q$.
Assuming the foliation is transversely orientable (i.e. that the normal bundle $\nuFF:=\tBun M/\tFF$ is an orientable vector bundle), the subbundle $\tFF$ can alternatively be regarded as the kernel of a nowhere vanishing, decomposable $q$-form $\omega = \omega^1 \w \ldots \w \omega^q$ on $M$ that is \emph{integrable}, in the sense that there exists a 1-form $\eta$ for which  $\d\omega=\eta\w\omega$.

The historical starting point for the characteristic class theory of foliations was the discovery of Godbillon and Vey \cite{gv} that, for codimension-1 regular foliations, the associated \( 3 \)-form $\eta\w \d\eta$ is closed, and its class in de Rham cohomology (now called the \emph{Godbillon-Vey class}) depends only on \( \FF \), and not on the choices of $\eta$ and $\omega$.
Nontriviality of the Godbillon-Vey class was first shown by Roussarie (also in \cite{gv}), while in \cite{thurston1} Thurston gave a construction exhibiting continuous variation of the Godbillon-Vey class for a family of regular foliations of the 3-sphere.

It was discovered by Bott \cite{bott2} that the Godbillon-Vey class of regular codimension-\( 1 \) foliations is  the simplest example of a large family of so-called \emph{secondary characteristic classes} associated to regular foliations of arbitrary codimension, transversely orientable or not.
More precisely, Bott showed in \cite{bott1} that the normal bundle of any regular codimension-\( q \) foliation admits connections that are flat along leaves (now called \emph{Bott connections}), and proved as an easy consequence that any Pontryagin polynomial of degree more than \( 2q \) vanishes when evaluated on the curvature of a Bott connection.
Bott deduced that certain Chern-Simons \cite{chernsimons} transgression forms of such polynomials-in-curvature therefore define de Rham cohomology  classes; in particular, for a codimension-1 regular foliation, the Godbillon-Vey class arises as a transgression of the square of the first Pontryagin form.

Bott's Chern-Weil account of the characteristic classes of regular foliations coincided with deep work by Gel'fand and Fuks studying the (continuous) cohomology of infinite-dimensional Lie algebras of vector fields \cite{gelfuks1, gelfuks2, gelfuks3, gelfuks}.  Gel'fand and Fuks in particular computed the $\O(q)$-basic cohomology $H^{*}(A(\af_{q}),\O(q))$ of the Lie algebra $\af_{q}$ of $\infty$-jets at zero of vector fields on $\RB^{q}$, which is now frequently referred to simply as the ``Gel'fand-Fuks cohomology".  They discovered that, in addition to encoding the usual Pontryagin classes of real vector bundles, this Lie algebra cohomology automatically encodes Bott's vanishing theorem and the consequent secondary characteristic classes, and is isomorphic to the cohomology of a well-understood finite-dimensional subalgebra $WO_{q}\hookrightarrow A^{*}(\af_{q})$, whose inclusion factors through the 2-jets of vector fields.
The relationship between Gel'fand-Fuks cohomology and the characteristic classes of foliations was formalised by Bott and Haefliger \cite{botthaef, bott3}, drawing on work of Kobayashi regarding higher order frame bundles \cite{kob}.

\noindent\textbf{Functoriality and Haefliger structures:} One of the defining properties of vector bundle characteristic classes is functorialiality under pullbacks.
One must expect the same to hold for the characteristic classes of foliations, but immediately the problem arises that a smooth map $\phi \colon N\rightarrow M$ into a regularly foliated manifold  \( M \) does not in general pull back a regular foliation.
More precisely, by the Frobenius theorem, a regular foliation of $M$ is determined by an open covering $\UU:=\{U_{\alpha}\}_{\alpha\in\AF}$ together with submersive first integrals $f_{\alpha} \colon U_{\alpha}\rightarrow\RB^{q}$ whose level sets define the leaves of the foliation in $U_{\alpha}$.
The ``pullback foliation" of $N$, defined by the open covering $\{\phi^{-1}(U_{\alpha})\}_{\alpha\in\AF}$ and functions $f_{\alpha}\circ\phi \colon \phi^{-1}(U_{\alpha})\rightarrow\RB^{q}$, admits singularities wherever $\phi$ is not transverse to leaves.
So, that the characteristic classes of foliations be functorial necessitates their definition for singular foliations of this nature.

This lack of functoriality was solved by Haefliger, whose insight in \cite{ha3} was to categorify codimension-$q$ foliations, regular or not, by considering the \emph{transition functions} between local first integrals as fundamental.  Since Haefliger's insight is essential to our work, we record his definition here.

\begin{defn}\label{haefycocycle}
	Let $X$ be a topological space.  A \textbf{Haefliger cocycle of codimension $q$} on $X$ consists of an open covering $\{U_{\alpha}\}_{\alpha\in\AF}$ of $X$ together with continuous functions $f_{\alpha} \colon U_{\alpha}\rightarrow\RB^{q}$, called \textbf{Haefliger charts},
	and for each $x\in U_{\alpha}\cap U_{\beta}$, the \textbf{transition function} $h_{\alpha\beta}^{x}$, a local diffeomorphism  of $\RB^{q}$ defined on a neighbourhood of $f_{\beta}(x)$, such that for all indices \( \alpha, \beta \) and \( x\in U_{\alpha}\cap U_{\beta} \),
	\begin{enumerate}
		\item the assignment $x\mapsto h^{x}_{\alpha\beta}$ is continuous, in the sense that the map $(y,\vec{s})\mapsto h^{y}_{\alpha\beta}(\vec{s})$ is continuous for \( (y, \vec{s}) \) near \( (x,f_\beta(x)) \in X \times \RB^q \),
		\item $f_{\alpha} = h^{x}_{\alpha\beta}\circ f_{\beta}$ on some open neighbourhood of $x$ in \( U_{\alpha}\cap U_{\beta} \), and
		\item the transition functions satisfy the following cocycle condition: for all $x\in U_{\alpha}\cap U_{\beta}\cap U_{\delta}$, one has $h^{x}_{\alpha\beta}\circ h^{x}_{\beta\delta} = h^{x}_{\alpha\delta}$ as germs.
	\end{enumerate}
	Two Haefliger cocycles are equivalent if there exists a third Haefliger cocycle refining them, and a \textbf{Haefliger structure} is an equivalence class of Haefliger cocycles.
\end{defn}

Of particular concern to Haefliger were the \emph{numerable} Haefliger structures, those which may be represented by a cocycle defined over an open cover admitting a subordinate partition of unity.
This class includes all Haefliger structures on smooth manifolds and, in particular, all regular foliations.
Haefliger showed that the category $\HH^{q}$, whose objects are numerable Haefliger structures of codimension $q$ and whose morphisms are continuous maps pulling back the structure on the codomain to that on the domain, admits a terminal object $(B\Gamma_{q},\gamma)$ \cite[Theorem 7]{ha6}, the \emph{classifying space} of codimension-$q$ Haefliger structures (here $\Gamma_{q}$ is the groupoid of germs of local diffeomorphisms of $\RB^q$).  As an application of this fact, he showed that there exists a universal characteristic map $H^{*}(A(\af_{q}),\O(q))\rightarrow H^{*}(B\Gamma_{q})$ from Gel'fand-Fuks cohomology to the cohomology of \( (B\Gamma_{q},\gamma) \), thereby furnishing a characteristic map for \emph{all} numerable Haefliger structures of codimension $q$ \cite[p. 53]{ha4}, which is functorial under morphisms in $\HH^q$.

\noindent\textbf{A gap in the theory:}
Let $\TT$ denote the category whose objects are pairs $(X,c)$, where $X$ is a topological space and $c$ is a cohomology class on \( X \), and whose morphisms are continuous maps that pull back the class of the codomain to that of the domain.
Haefliger's characteristic map furnishes an assignment $\HH^{q} \times A^{*}(\af_{q})_{\O(q)} \rightarrow \TT$, which is functorial
in the first argument and homomorphic (with respect to the ring structures of $A^{*}(\af_{q})_{\O(q)}$ and of cohomology) in the second argument.

Now, denote by $\FS^{q}_{reg}$ the category of regular codimension-\( q \) foliations on smooth manifolds, with morphisms the regular foliation maps.
The Frobenius theorem induces an inclusion $\FS^{q}_{reg}\hookrightarrow\HH^{q}$ of categories, and Bott's Chern-Weil homomorphism amounts to a similarly functorial-homomorphic assignment $\FS^{q}_{reg} \times WO_{q} \rightarrow\TT$.
Bott proves in \cite[Theorem 10.16]{bott2} that his Chern-Weil characteristic map is compatible with that of Haefliger, in the sense that it recovers the characteristic classes coming from Haefliger's classifying space.  Put another way, Bott's Chern-Weil characteristic map makes the diagram
\begin{equation}\label{introdiagram}
\begin{tikzcd}
\FS^{q}_{reg} \times WO_{q}\ar[dr] \ar[rr, hook] & & \HH^{q} \times A^{*}(\af_{q})_{\O(q)} \ar[dl] \\ & \TT &
\end{tikzcd}
\end{equation}
commute.

On the one hand, the category $\FS^{q}_{reg}$ is not sufficiently large to accurately model the full subcategory of smooth manifolds in $\HH^{q}$, even up to homotopy, because there exist smooth Haefliger structures not homotopic to a regular foliation.\footnote{For example, any choice of \( q \) smooth functions on a manifold defines a smooth Haefliger structure.
But on \( S^{2} \), no codimension-\( 1 \) Haefliger structure is homotopic to a regular foliation, because \( S^2 \) admits no regular codimension-\( 1 \) foliation.}
For such Haefliger structures, Bott's Chern-Weil theory offers no way to probe topology.
On the other hand, while Haefliger's characteristic map works for all Haefliger structures,  and the abstract definition allows quick proofs of general properties, this very abstraction makes it difficult to use differential geometric methods for the concrete study of topology, as is typically necessary for finer-grained results \cite{gv, thurston1}.
A middle ground is needed, one permitting the construction of characteristic classes for all Haefliger structures on manifolds, yet in terms of differential geometry.
This gap in the literature has stood for almost five decades.

\noindent\textbf{Closing the gap:}
Our primary result is the provision of such a middle ground. Our solution is an extension of Bott's Chern-Weil theory to a natural class of singular foliations, which we now define, restricting ourselves to \emph{smooth} Haefliger cocycles on smooth manifolds (those whose Haefliger charts and transition functions are smooth).

\begin{defn}\label{def1}
	The \textbf{singular set} \( \Sigma \) of a smooth Haefliger cocycle of codimension $q$ on a manifold $M$ of dimension at least $q$ is the closed set of all critical points of each $f_{\alpha}$;
	the singular set depends only on the associated Haefliger structure.
	The \textbf{regular set} is the complement of \( \Sigma \), which we often denote as \(\tilde{M}:= M -  \Sigma \).
	If \(\tilde{M} \) is dense in \( M \), we say that the Haefliger cocycle determines a \textbf{Haefliger-singular foliation $(M,\FF)$ of codimension $q$}, whose \textbf{regular subfoliation} is the regular foliation of $\tilde{M}$ determined by the restriction of the Haefliger cocycle to $\tilde{M}$.
\end{defn}
We note that the term ``singular foliation" is often used to refer to \emph{Stefan-Sussmann} singular foliations, namely those defined by integrable families of vector fields \cite{stefan, sussmann}.  We show in the appendix that every Haefliger-singular foliation (in fact, every smooth Haefliger structure) does define a Stefan-Sussmann singular foliation.  However, the converse is not true; for instance, the foliation of $\RB^2$ by the integrals of $x\partial_{x}+y\partial_{y}$ is Stefan-Sussmann, but admits no nontrivial first integrals in any neighbourhood of the origin, so cannot be associated to a Haefliger structure.

The category $\FS^{q}_{sing}$ of codimension-$q$ Haefliger-singular foliations, whose morphisms are smooth maps pulling back the foliation of the codomain to that of the domain, fits snugly in between $\FS^{q}_{reg}$ and $\HH^{q}$.  In addition, we describe below a Chern-Weil map extending that of Bott to Haefliger-singular foliations, expanding the commuting diagram \eqref{introdiagram} to
\begin{equation}
\begin{tikzcd}
\FS^{q}_{reg} \times WO_{q}\ar[dr] \ar[r, hook] & \FS^{q}_{sing} \times WO_{q} \ar[d] \ar[r,hook] & \HH^{q} \times A^{*}(\af_{q})_{\O(q)} \ar[dl] \\ & \TT. &
\end{tikzcd}
\end{equation}
We prove that $\FS^{q}_{sing}$ is large enough to fill the gap described above.
Every object in $\HH^{q}$ is homotopic to a Haefliger-singular foliation, subject only to the obvious constraint that the underlying manifold have dimension at least $q$, and every morphism in $\HH^{q}$ between two such objects is homotopic to a morphism in $\FS^{q}_{sing}$.
In this way, our theory greatly expands the application of the Chern-Weil theory of Gel'fand-Fuks characteristic classes of Haefliger structures.
As a straightforward application of our theory, we describe how the classical algorithm for the geometric construction of the Godbillon-Vey invariant, used extensively for the study of regular foliations \cite{gv, thurston1}, extends cleanly to the singular setting.

\noindent\textbf{Outline of paper:}
Section \ref{s2} consists of a review of the necessary prerequisites, on Gel'fand-Fuks cohomology and on the Chern-Weil approach to characteristic classes for regular foliations.
Section \ref{univ} covers a smooth (diffeological) generalisation of Haefliger's classifying theorem, and an associated de Rham-theoretic universal characteristic map.
This is the most complete and explicit account of the universal characteristic map that we are aware of, enabling a clean proof that our Chern-Weil construction recovers the correct characteristic classes.

Section \ref{haefybundle} gives a de Rham-theoretic proof that the Chern-Weil approach to regular foliations recovers the characteristic classes coming from Haefliger's classifying space.
Although the result, originally due to Bott \cite{bott2}, has been known for some decades, our presentation is novel in that it invokes an under-utilised family of fibre bundles, that we call the \emph{Haefliger bundles}, which are defined over both regular and singular foliations.
The Haefliger bundles act as an intermediary between the universal characteristic map and the Chern-Weil characteristic map for regular foliations.

Section \ref{singularCHW} consists of an extension of these ideas to singular foliations.
We characterise the metric-Bott-connection pairs on the normal bundle of the regular subfoliation of a Haefliger-singular foliation that are \emph{adapted} to the singularities via the jets of their exponential maps.  Such pairs we term \emph{adapted geometries}. We then prove our Chern-Weil theorem for singular foliations with adapted geometry.

\begin{thm}\label{thm: adapted CW map}
	Let $M$ be a Haefliger-singular foliation of codimension $q$.  Any adapted geometry on $M$ specifies a unique Chern-Weil homomorphism from $WO_{q}\subset A^{*}(\af_{q})$ to $\Omega^{*}(M)$.
	The Chern-Weil homomorphism descends to cohomology to agree with the Haefliger characteristic map from Gel'fand-Fuks cohomology.
\end{thm}

We prove in addition that adapted geometries and associated Chern-Weil homomorphisms are functorial under \emph{Haefliger-singular maps}, namely those smooth maps which pull back a Haefliger-singular foliation of the codomain to a Haefliger-singular foliation of the domain.
The functoriality is used in our second main result, that Haefliger-singular foliations with adapted geometries suffice to recover the characteristic classes of all Haefliger structures on manifolds:

\begin{thm}
	All Haefliger-singular foliations admit adapted geometries.  Moreover, the category $\FS^{q}_{sing}$ of codimension-$q$ Haefliger-singular foliations with adapted geometries is homotopy-equivalent to the category $\HH^{q}_{man}$ consisting of codimension-$q$ Haefliger structures on smooth manifolds of dimension at least $q$.
\end{thm}

Section \ref{singularCHW} is concluded with an application of our theory: we generalise the classical algorithm for the construction of the Godbillon-Vey invariant, famously used by Thurston to study the topology of $B\Gamma_{q}$ \cite{thurston2}, to Haefliger-singular foliations.

Finally, we include 3 appendices, which support the statements of the main part of the paper, and may be of interest on their own.

Appendix \ref{app: Stefan-Sussman} explains the relationship between Haefliger structures and Stefan-Sussman foliations, with a demonstration that every Haefliger structure is an example of a Stefan-Sussman foliation.

In Appendix \ref{app: smooth Haefliger classifying} we carefully extends Haefliger's construction of the classifying space for Haefliger structures to the smooth category, using diffeologies as the category of generalized smooth spaces.
This is a functorial construction, along classical lines, assigning to a smooth groupoid \( \Gamma \) the space \( B\Gamma \) that classifies smoothly numerable \( \Gamma \)-structures on smooth spaces.

Appendix \ref{app: sections of Haefliger bundles} is concerned with sections of the Haefliger bundle quotient by \( \O(q) \) and their uniqueness up to homotopy.
This is a classical statement for continuous sections, but it is necessary to show that smooth (diffeological) sections exist and are all smoothly homotopic to each other.
The smooth statements are used to define the smooth characteristic map, and underlie the fact that it is well-defined up to homotopy.

\subsection{Acknowledgements}
LM was supported by the Australian Research Council Discovery Project grant DP200100729.
BM was supported by the Australian Research Council Discovery Project grant DP190102360.
LM wishes to thank Iakovos Androulidakis, Moulay Benameur and David Roberts for helpful discussions.
BM wishes to thank Mike Eastwood and Thomas Leistner for helpful discussions.

\section{Background}\label{s2}

We work under the convention that the set of natural numbers $\NB$ includes zero.
All manifolds are assumed to be connected, paracompact, Hausdorff, and without boundary, unless otherwise stated.

\subsection{Gel'fand-Fuks cohomology}

Given a Lie group $G$ with Lie algebra $\gf$, a \emph{$G$-differential graded algebra}, or \( G \)-DGA, is a differential graded \( \RB \)-algebra $(A^{\bullet},\d)$, where $A^{\bullet}$ carries a smooth, degree 0 action of $G$, with associated Lie derivative defined by
\[
L_{\xi}(a):=\frac{d}{dt}\bigg|_{0}\exp(t\xi)\cdot a, \qquad \xi\in\gf, \quad a\in A,
\]
and a contraction operator $\iota_{\xi} \colon A^{\bullet}\rightarrow A^{\bullet-1}$ for each $\xi\in\gf$ such that $\d$ and $\iota$ anticommute up to the Lie derivative,
\[
L_{\xi} = \d \iota_{\xi}+\iota_{\xi} \d.
\]
Given any Lie subgroup $K$ of $G$, with Lie algebra $\kf$, the \emph{$K$-basic subalgebra}  of \( A^{\bullet} \) is the differential graded subalgebra
\[
A^{\bullet}_{K}:=\bigl\{ a\in A^{\bullet} \colon k\cdot a = a\mbox{ for all }k\in K, \mbox{ and }\iota_{\xi}a = 0\mbox{ for all }\xi\in\kf \bigr\} .
\]
Denote the cohomology of \( A^{\bullet}_{K} \) by $H^{*}(A,K)$.

For $0\leq q,k<\infty$, fix the Lie group $G^{k}_{q}$, the \emph{$k^{th}$-order jet group of $\RB^{q}$}, comprising the $k$-jets at zero of diffeomorphisms of $\RB^{q}$ that fix zero.
Let $\gf^{k}_{q}$ denote the Lie algebra of \( G^{k}_{q} \).
As \( k \) ranges from \( 0 \) to \( \infty \), the natural maps between the $G^{k}_{q}$ form a projective system of Lie groups.
The projective limit $G^{\infty}_{q}$ of this system is the \emph{infinite order jet group of $\RB^{q}$}, which inherits a natural smooth structure \cite[Section 7.1]{saunders}.
With this smooth structure, $G^{\infty}_{q}$ is an infinite-dimensional Lie group, with Lie algebra $\gf^{\infty}_{q}$ equal to the projective limit of the projective system of Lie algebras determined by the $\gf^{k}_{q}$.

For $q\geq 1$, denote by $\af_{q}$ the Lie algebra of $\infty$-jets at zero of smooth vector fields on $\RB^{q}$.
Endow $\af_{q}$ with a projective limit topology by identifying it as the projective limit of the system $\{\af_{q}^{k},\pi_{k}\}$ of finite-dimensional manifolds of $k$-jets  $\af_{q}^{k}$ at zero of vector fields on $\RB^{q}$, with $\pi_{k} \colon \af^{k}_{q}\rightarrow\af^{k-1}_{q}$ the canonical projection.
This, plus the Lie bracket on \( \af_{q} \) determined by that of vector fields, determines a topological Lie algebra structure on \( \af_{q} \) \cite{gelfuks}.

For any $X \in \af_{q}$, the flow \( \varphi_{t} \) of \( X \) is well-defined near \( 0 \) for small \( t \).
As such, \( X \) can be realised as the time derivative of a path of \( \infty \)-jets of diffeomorphisms,
\begin{equation}\label{timeder}
  X = \frac{d}{dt}\bigg|_{0}j^{\infty}_{\vec{0}}(\varphi_{t}) .
\end{equation}
This gives, in particular, a natural inclusion $\gf^{\infty}_{q}\hookrightarrow\af_{q}$ of Lie algebras, with image characterised as the jets of those vector fields that vanish at \( 0 \).

For $k\geq 1$, denote by $A^{k}(\af_{q})$ the space of continuous, alternating, multi-linear functionals $\Wedge^{k}\af_{q}\rightarrow\RB$.
The usual Chevalley-Eilenberg formula defines a differential $\d \colon A^{k}(\af_{q})\rightarrow A^{k+1}(\af_{q})$,
given for $c\in A^{k}(\af_{q})$ and $X_{0},\dots,X_{k}\in\af_{q}$ by
\[
\d c(X_{0},\dots,X_{k}):=\sum_{i<j}(-1)^{i+j}c([X_{i},X_{j}],X_{0},\dots,\hat{X}_{i},\dots,\hat{X}_{j},\dots,X_{k+1}) .
\]
Now, $A^{*}(\af_{q})$ is a $G^{\infty}_{q}$-differential graded algebra.
Indeed, the inclusion $\gf^{\infty}_{q}\hookrightarrow\af_{q}$ defines a contraction operator $\iota_{\xi} \colon A^{*}(\af_{q})\rightarrow A^{*-1}(\af_{q})$ for all $\xi\in\gf^{\infty}_{q}$,
and the right action of $G^{\infty}_{q}$ on $\af_{q}$ defined by
\[
  X \cdot j^{\infty}_{\vec{0}}(g) = \frac{d}{dt}\bigg|_{0}j^{\infty}_{\vec{0}}\big(g^{-1}\circ\varphi_{t}\circ g\big), \qquad
  X = \frac{d}{dt}\bigg|_{0}j^{\infty}_{\vec{0}}(\varphi_{t})\in\af_{q},
  \quad
  j^{\infty}_{\vec{0}}(g)\in G_{q}^{\infty},
\]
induces an action of $G^{\infty}_{q}$ on $A^{*}(\af_{q})$, compatible with the contraction operator.
The $\O(q)$-basic cohomology $H^{*}(A(\af_{q}),\O(q))$ of $A^{*}(\af_{q})$ is frequently referred to as simply the \emph{Gel'fand-Fuks cohomology}, and was computed by Gel'fand and Fuks in \cite{gelfuks}.

We now review the method outlined by Bott \cite{bott3} for the computation of Gel'fand-Fuks cohomology $H^{*}(A(\af_{q}),\O(q))$, and in the next subsection, how the $\O(q)$-basic $A^{*}(\af_{q})$-cocycles define characteristic classes.
Fix the standard linear coordinates \( s^i \) on \( \RB^q \).
For any multi-index $\alpha\in\NB^{q}$ and for $1\leq i\leq q$, the Dirac-derivative functional on \( \af_{q} \) defined by
\[
\delta^{i}_{\alpha}(X):=(-1)^{|\alpha|}\frac{\partial^{|\alpha|}X^{i}}{\partial s^{\alpha}}(0), \qquad X \in \af_{q}
\]
is an element of $A^{1}(\af_{q})$, and by elementary distribution theory, the collection of all such functionals generates $A^{1}(\af_{q})$ linearly, hence generates $A^{*}(\af_{q})$ as an algebra.
The following structure equations then follow from a routine calculation:
\begin{equation}\label{structureequations}
\d\delta^{i}+\delta^{i}_{j}\w\delta^{j} = 0,\qquad\qquad \d\delta^{i}_{j}+\delta^{i}_{jk}\w\delta^{k} +\delta^{i}_{k}\w\delta^{k}_{j} = 0,
\end{equation}
with the Einstein summation convention assumed.
As we will recall in the next subsection, these are the universal structure equations for torsion-free affine connections.
The first is the structure equation for torsion-free affine connections (with $\delta^{i}$ the components of the solder form, and $\delta^{i}_{j}$ the components of the connection form), while the second is the equation defining the curvature $\Delta^{i}_{j}:=-\delta^{i}_{jk}\w\delta^{k} = \d\delta^{i}_{j} + \delta^{i}_{k}\w\delta^{k}_{j}$ of the connection.

Define the \( q \times q \) matrices of \( 1 \)- and \( 2 \)-forms
\begin{equation}\label{deltaDelta}
	\delta:=\bigl(\delta^{i}_{j}\bigr),\qquad\Delta:=\bigl(\Delta^{i}_{j}\bigr), \qquad i,j = 1, \ldots, q
\end{equation}
corresponding to the functionals $\delta^{i}_{j}$ and $\Delta^{i}_{j}$ respectively, and denote by $\delta:=\delta_{s}+\delta_{o}$ and $\Delta:=\Delta_{s}+\Delta_{o}$ their respective decompositions into symmetric and antisymmetric components.

Let $\RB[c_{1},\dots,c_{q}]_{q}$ denote the quotient of the polynomial algebra generated by symbols $c_{i}$ of degree $2i$ by the ideal of elements of degree greater than $2q$, and let $\Wedge(h_{1},h_{3},\dots,h_{l})$ denote the exterior algebra generated by symbols $h_{i}$ of degree $2i-1$, with $l$ the largest odd integer that is less than or equal to $q$.  Equip the graded-commutative algebra
\[
WO_{q}:=\RB[c_{1},\dots,c_{q}]_{q}\otimes\Wedge(h_{1},h_{3},\dots,h_{l})
\]
with the differential $\d$ defined by $\d c_{i} = 0$ for all $i$ and $\d h_{j} = c_{j}$ for all $j$ odd.
Then by the results of \cite{gue}, $WO_{q}$ embeds as a differential graded subalgebra of $A^{*}(\af_{q})_{\O(q)}$ according to the formul\ae
\begin{equation}\label{c_i}
c_{i}:=\Tr(\Delta^{\w i}),\qquad 1\leq i\leq q,
\end{equation}
\begin{equation}\label{hj}
h_{j}:=j\Tr\bigg(\int_{0}^{1}\delta_{s}(t\Delta_{s}+\Delta_{o}+(t^{2}-1)\delta_{s}^{\w 2})^{\w(j-1)}dt\bigg),\qquad 1\leq j\leq q,\,j\text{ odd}.
\end{equation}
One has the following theorem, which computes the cohomology of the infinite-dimensional algebra $A^{*}(\af_{q})_{\O(q)}$ in terms of the finite-dimensional subalgebra $WO_{q}$.

\begin{thm}\cite[Theorem 2]{botthaef}\label{woq}
	The algebra inclusion $WO_{q}\hookrightarrow A^{*}(\af_{q})_{\O(q)}$ induces an isomorphism on cohomology, $H^{*}(WO_{q})\cong H^{*}(A(\af_{q}),\O(q))$.
\end{thm}

\subsection{Frame bundles and tautological forms}

Let $(M,\FF)$ be a regular foliation of codimension $q$.
For \( x \in M \), a \emph{transverse embedding through \( x \)} is an embedding $u\colon \RB^{q}\rightarrow M$ such that \( u(0) = x \) and for each $\vec{s}\in\RB^{q}$ one has $\tBun_{u(\vec{s})}M = du_{\vec{s}}(\tBun_{\vec{s}}\RB^{q})\oplus \tBun_{u(\vec{s})}\FF$.
Two $k$-jets $j^{k}_{\vec{0}}(u_{1})$ and $j^{k}_{\vec{0}}(u_{2})$ of transverse embeddings through $x$ are said to be \emph{leaf space equivalent} if for some (hence any) Haefliger chart $f\colon U\rightarrow\RB^{q}$ defined around $x$ one has $j^{k}_{\vec{0}}(f\circ u_{1}) = j^{k}_{\vec{0}}(f\circ u_{2})$.  The leaf space equivalence class of a $k$-jet $j^{k}_{\vec{0}}(u)$ will be denoted $j^{k}_{\vec{0},\t}(u)$ (the $\t$ being used to denote \emph{transverse}).

\begin{defn}\cite{gcc}\label{framebundle}
	Let $(M,\FF)$ be a regular foliation of codimension $q$, and let $0\leq k\leq \infty$.
	The \textbf{transverse $k$-frame bundle} of $(M,\FF)$ is the principal $G^{k}_{q}$-bundle $\Fr_{k}(M/\FF)\rightarrow M$ whose fibre over $x\in M$ is the set of leaf space equivalence classes of $k$-jets at $\vec{0}$ of transverse embeddings through \( x \).
\end{defn}

\begin{ex}\label{ex-frame characteristic map}
        If \( q = \dim(M) \), then $\FF$ is a regular foliation of $M$ by points, Haefliger charts are coordinate charts, and this definition recovers the standard definition of the $k$-frame bundle $\Fr_{k}(M)$ of $M$.
In this case, the diffeomorphism group $\Diff(M)$ acts from the left on $\Fr_{k}(M)$ by postcomposition, and this action commutes with the principal right action of $G^{k}_{q}$.
Denote by $\Omega^{*}(\Fr_{k}(M))^{\Diff(M)}$ the $G^{k}_{q}$-DGA of $\Diff(M)$-invariant forms on \( \Fr_{k}(M) \).
Using Equation \eqref{timeder}, the \emph{tautological \( \af^{k-1}_{q} \)-valued \( 1 \)-form} $\omega^{k}$ defined by the formula
\[
\omega^{k}_{j^{k}_{\vec{0}}(u)}\bigg(\frac{d}{dt}\bigg|_{0}j^{k}_{\vec{0}}(u_{t})\bigg):=\frac{d}{dt}\bigg|_{0}j^{k-1}_{\vec{0}}(u^{-1}\circ u_{t}) \qquad \mbox{ for \ } \frac{d}{dt}\bigg|_{0}j^{k}_{\vec{0}}(u_{t})\in \tBun_{j^{k}_{\vec{0}}(u)}\Fr_{k}(M)
\]
is easily seen to be an element of $\Omega^{1}(\Fr_{k}(M))^{\Diff(M)}\otimes\af^{k-1}_{q}$.
The tautological 1-forms $\omega^{k}$ were introduced by Kobayashi \cite{kob}.

The $\af_{q}$-valued 1-form $\omega:=\omega^{\infty}$ on \( \Fr_{\infty}(M) \) satisfies the Maurer-Cartan identity 
\[ d\omega^{\infty}+\tfrac{1}{2}[\omega^{\infty},\omega^{\infty}] = 0 \]
(see \cite[p. 113]{gcc}), and the identities in Equation \eqref{structureequations} are low degree components of this equation.
Furthermore, \( \omega \) has trivial kernel, and defines a canonical trivialisation of the tangent bundle, $\tBun\Fr_{\infty}(M)\cong\Fr_{\infty}(M)\times\af_{q}$.
As such, any $c\in A^{*}(\af_{q})$ defines a form $c(\omega \w \cdots \w \omega)\in\Omega^{*}(\Fr_{\infty}(M))^{\Diff(M)}$, and this assignment gives rise to a canonical isomorphism
\begin{equation}\label{omega}
  \begin{tikzcd}
    A^{*}(\af_{q}) \ar[r, "\omega"] & \Omega^{*}(\Fr_{\infty}(M))^{\Diff(M)}
  \end{tikzcd}
\end{equation}
of $G^{\infty}_{q}$-differential graded algebras (cf. \cite[Proposition 3.5]{folbund}).

It is known \cite[p. 131, Proposition]{natopdiffgeom} that for each finite $k\geq 1$, the natural projection $G^{k}_{q}\rightarrow \GL(\RB^{q})$ is a principal fibre bundle, with typical fibre a contractible nilpotent Lie group whose Lie exponential map is a global diffeomorphism.
The same is therefore also true of the fibration $G^{\infty}_{q}\rightarrow\GL(\RB^{q})$.  It follows that $\Fr_{\infty}(M)\rightarrow \Fr_{1}(M)$ has contractible fibres, and always admits sections.

In particular, letting $\nabla$ be an affine connection on $M$, with exponential map $\exp^{\nabla}$, the formula
\begin{equation}\label{section}
	\sigma_{\nabla}(\vec{e}_{x}):=j^{\infty}_{\vec{0}}\big(\vec{s}\mapsto\exp^{\nabla}_{x}(\vec{s}\cdot\vec{e}_{x})\big) \qquad \vec{e}_{x} \in \Fr_{1}(M)
\end{equation}
defines a \( G^{1}_{q} \)-equivariant section $\sigma_{\nabla}$ of $\Fr_{\infty}(M)\rightarrow \Fr_{1}(M)$.  Recall now the functionals $\delta$ and $\Delta$ introduced in Equation \eqref{deltaDelta}.

\begin{prop}\cite[Lemma 18]{diffcyc}\label{introprop}
	Let $\nabla$ be an affine connection on a manifold $M$, and let $\sigma_{\nabla} \colon \Fr_{1}(M)\rightarrow\Fr_{\infty}(M)$ be the section given in Equation \eqref{section}.  Then $(\sigma_{\nabla})^{*}\delta(\omega)\in\Omega^{1}(\Fr_{1}(M);\mathfrak{gl}(\RB^{q}))$ is the connection form associated to $\nabla$ and $(\sigma_{\nabla})^{*}\Delta(\omega)\in\Omega^{2}(\Fr_{1}(M);\mathfrak{gl}(\RB^{q}))$ is its curvature.
\end{prop}

It follows immediately from Proposition \ref{introprop} that the forms $(\sigma_{\nabla})^{*}c_{i}(\omega)\in\Omega^{2i}(\Fr_{1}(M))_{\GL(\RB^{q})}$ from Equation \eqref{c_i} are the Pontryagin forms of $M$ defined with respect to the curvature tensor of $\nabla$, which displays the fundamental relationship between Chern-Weil theory and Gel'fand-Fuks cohomology.
Moreover, since the $c_{i}$ and $h_{i}$ only depend on 2-jets, and $WO_{q} \hookrightarrow A(\af_{q})$ is a quasi-isomorphism, one sees that in cohomology it suffices to work with $\omega^{2}\in\Omega^{1}(\Fr_{2}(M))^{\Diff(M)}\otimes\af_{q}^{1}$.
\end{ex}

The facts of Example \ref{ex-frame characteristic map} generalise to the transverse frame bundles of foliations.  Specifically, the transverse frame bundles of a regular foliation admit tautological forms, which model geometric structures on the normal bundle of the foliation.  This is folklore, but will be made precise after our introduction of Haefliger bundles in Section \ref{univ}.

\subsection{Chern-Weil for regular foliations}

In the early nineteen-seventies, R. Bott showed that for a regular foliation $(M,\FF)$ of codimension $q$, characteristic classes can be obtained in the following manner \cite{bott2}.
Regard the normal bundle $\nuFF := \tBun M/\tFF$ of $\FF$ as a subbundle of $\tBun M$ that is complementary to $\tFF$. (For instance, the orthogonal complement of $\tFF$ with respect to a Riemannian metric on $M$.)
For a vector field $Z$ on $M$, denote by $Z= Z_{\FF}+Z_{\nu}$ its decomposition into leafwise and normal components respectively.

Bott defined in \cite{bott1} connections $\nabla$ on $\nuFF$, now called \emph{Bott connections}, that satisfy the equation
\[
  \nabla_{X} Y = \nabla_{X_{\nu}}Y + [X_{\FF},Y]_{\nu}, \qquad X\in\XF(M),\,Y\in\Gamma(\nuFF)\subset\XF(M) .
\]
The \emph{torsion} of a Bott connection $\nabla$ is the \( \nuFF \)-valued \( 2 \)-form
\( T \) defined by
\[
  T(X,Y):=\nabla_{X} Y_{\nu} - \nabla_{Y} X_{\nu} - [X,Y]_{\nu}
\]
for all $X,Y\in\XF(M)$.  The connection $\nabla$ is said to be \emph{torsion-free} if $T(X,Y) = 0$ for all $X,Y\in\XF(M)$.

Bott connections are flat along leaves, in that their curvature forms $R_{\nabla}$ vanish on restriction to the tangent distribution of leaves.  As a consequence, one has \emph{Bott's vanishing theorem}: for any Bott connection $\nabla$ on $\nuFF$, the associated Chern-Weil characteristic map $\RB[c_{1},\dots,c_{q}]\rightarrow\Omega^{*}(M)$ encoding the Pontryagin classes of $\nuFF$, defined on generators by
\[
  c_{i}\mapsto \Tr(R_{\nabla}^{\w i}),
\]
vanishes on all monomials (in the $c_{i}$) of degree greater than $2q$.

Let $\nabla^{1}$ be a Bott connection on $\nuFF$ and let $\nabla^{0}$ be a connection on $\nuFF$ that is compatible with some Euclidean metric on $\nuFF$.
For $t\in[0,1]$, denote by $\nabla^{t}:=t\nabla^{1}+(1-t)\nabla^{0}$ the affine combination of $\nabla^{0}$ and $\nabla^{1}$.
Bott's vanishing theorem has the following consequence.

\begin{thm}[Bott-Chern-Weil characteristic map]\cite[p. 67-69]{bott2}\label{bottchernweil}
  The map $\lambda_{\nabla^{1},\nabla^{0}} \colon WO_{q}\rightarrow\Omega^{*}(M)$ defined on generators by the formul\ae
  \begin{equation}\label{ci}
    \lambda_{\nabla^{1},\nabla^{0}}(c_{i}):=\Tr(R_{\nabla^{1}}^{\w i})\in\Omega^{2i}(M),
  \end{equation}
  \begin{equation}\label{hi}
    \lambda_{\nabla^{1},\nabla^{0}}(h_{i}):= i\int_{0}^{1}\Tr\big((\nabla^{1}-\nabla^{0})\w R_{\nabla^{t}}^{\w (i-1)}\big)\,dt\in\Omega^{2i-1}(M)
  \end{equation}
  extends to a homomorphism of differential graded algebras whose descent to cohomology does not depend on the choice of Bott connection or metric connection.
\end{thm}

In \cite{gue}, Guelorget gives an alternative construction of the characteristic map, which depends on a choice of metric and Bott connection on \( \nuFF \).
Our construction to define the characteristic map on Haefliger-singular foliations closely resembles Guelorget's, depending on special singular choices of a metric and Bott connection over the regular locus, so we recall the Guelorget characteristic map.

For a regular foliation \( (M, \FF) \), let \( \nabla \) be a Bott connection and \( \epsilon \) a metric on \( \nuFF \).
On the \( \epsilon \)-orthonormal frame bundle $\Fr_{O}(\nuFF)$ of \( \nuFF \), the Bott connection determines a connection form $\omega\in\Omega^{1}(\Fr_{O}(\nuFF);\mathfrak{gl}(\RB^{q}))$ and its curvature form \( \Omega \).
Let $\omega = \omega_{s}+\omega_{o}$ and $\Omega:=\Omega_{s}+\Omega_{o}$ be their decompositions into symmetric and antisymmetric components.
\begin{thm}[Guelorget characteristic map]\cite{gue}\label{guegorlet}
  The forms defined on $\Fr_{O}(\nuFF)$,
  \begin{equation}\label{cig}
    \lambda_{\nabla,\varepsilon}(c_{i}):=\Tr(\Omega^{\wedge i}),\qquad 1\leq i\leq q
  \end{equation}
  \begin{equation}\label{hig} \lambda_{\nabla,\varepsilon}(h_{i}):=i\Tr\bigg(\int_{0}^{1}\omega_{s}(t\Omega_{s}+\Omega_{o}+(t^{2}-1)\omega_{s}^{\wedge2})^{\wedge(i-1)}dt\bigg),\qquad i\leq q,\,\text{i odd} ,
  \end{equation}
  are $\O(q)$-basic, so descend to forms on $M\cong\Fr_{O}(\nuFF)/\O(q)$.
  The resulting map $\lambda_{\nabla,\varepsilon} \colon WO_{q}\rightarrow \Omega^{*}(M)$ is a homomorphism of differential graded algebras whose descent to cohomology is independent of the choice of Bott connection $\nabla$ and metric $\varepsilon$.
\end{thm}

Guelorget showed that her characteristic map coincides with the Bott-Chern-Weil map provided one makes careful choices.
To this end, a choice of Riemannian metric $g$ on $M$ can be used to construct the following:
\begin{enumerate}
\item an identification of $\nuFF$ with the orthogonal complement of $\tBun\FF$,
\item a metric $\epsilon$ on $\nuFF\subset\tBun M$,
\item a torsion-free Bott connection, the \emph{Bott Levi-Civita connection} $\nabla^{1}$, given by
  \[
    \nabla^{1}_{X} Y:=[X_{\FF},Y]_{\nu}+(\nabla^{LC}_{X_{\nu}} Y)_{\nu},
  \]
  where $\nabla^{LC}$ is the Levi-Civita connection associated to $g$, and
\item an \( \epsilon \)-compatible connection \( \nabla^{0} \), determined by taking the antisymmetric component of the connection form of \( \nabla^{1} \) relative any orthonormal framing.
\end{enumerate}
\begin{propdef}\cite[Remarque (c)]{gue}\label{gue2}
  With these choices of geometric data, the homomorphism $\lambda_{\nabla^{0},\epsilon}$ of Theorem \ref{guegorlet} is equal to the Bott-Chern-Weil homomorphism $\lambda_{\nabla^{1},\nabla^{0}}$ of Theorem \ref{bottchernweil}.
\end{propdef}

\section{Classifying spaces and the universal characteristic map}\label{univ}

We present in this section a diffeological%
\footnote{Diffeology is a categorically well-behaved formalisation of smoothness, in which both manifolds and classifying spaces can be placed on the same footing.
  A brief recap is provided in the appendix, or the reader may safely replace ``diffeological" with ``smooth" in all of what follows.}
version of the classical constructions of characteristic classes, both via universal classifying spaces and the Chern-Weil theory.
The diffeological language will make it straightforward to generalise these constructions to singular Haefliger structures in the following sections.
Full details are provided in Appendix \ref{app: smooth Haefliger classifying}.

\subsection{Classifying space for diffeological groupoids}

Recall that a \emph{groupoid} is a small category with inverses.  The range (or target) and source maps of a groupoid will always be denoted $r$ and $s$ respectively.
A groupoid is said to be \emph{diffeological} if its morphism set is a diffeological space, with object set equipped, via the unit map, with the subspace diffeology, and for which the range, source, composition and inversion maps are all smooth maps. A diffeological groupoid is \emph{\'{e}tale} if the source and target maps are local homeomorphisms with respect to the D-topology.
Two examples follow.

\begin{ex}[\v{C}ech groupoid]
	Let $M$ be a diffeological space and $\UU = \{U_{\alpha}\}_{\alpha\in\AF}$ an open cover of $M$.  The \emph{\v{C}ech groupoid} $\check{\UU}$ of $\UU$ is the groupoid whose morphism space is the disjoint union
	\[
	\check{\UU}^{(1)}:=\bigsqcup_{(\alpha,\beta)\in\AF^2} U_{\alpha}\cap U_{\beta} .
	\]
	An element of \( \check{\UU} \) is an ordered triple \( (x, \alpha, \beta) \) with \( x \in U_{\alpha} \cap U_{\beta} \), so that the source and target maps are defined by the rules  $s(x,\alpha,\beta) = (x,\beta,\beta)$ and $r(x,\alpha,\beta) = (x,\alpha,\alpha)$ respectively.
	In particular, the unit space \( \check{\UU} \) is the disjoint union of the degenerate intersections $U_{\alpha}\cap U_{\alpha}$.
	Composition is given by
	\[
	(x,\alpha,\beta)\cdot(x,\beta,\delta):=(x,\alpha,\delta).
	\]
	(Note the convention that morphisms act right to left.)
	As a disjoint union of diffeological spaces, $\check{\UU}$ inherits a canonical diffeology (the sum diffeology \cite[Section 1.39]{diffeology}), with respect to which it is an \'{e}tale diffeological groupoid.
\end{ex}

\begin{ex}[Haefliger groupoid]
	The \emph{Haefliger groupoid} $\Gamma_{q}$ is the groupoid of germs of local diffeomorphisms of $\RB^{q}$ \cite{ha3}.  For $\gamma\in\Gamma_{q}^{(1)}$ with source $x\in\RB^{q}$, representative local diffeomorphism $g \colon \dom(g)\rightarrow\RB^{q}$, and any open neighbourhood $U$ of $x$ contained in $\dom(g)$, define an open neighbourhood of \( \gamma \) by
	\[
	\NN(\gamma,g,U):=\{\germ_{x'}(g) \colon x'\in U\} .
	\]
	The collection of these for all \( (\gamma, g, U) \) determine the basis for a topology on $\Gamma_{q}$.
	This topology induces the standard topology on the unit space \( \RB^{q} \), with respect to which the range and source are local homeomorphisms.
	Therefore, the manifold structure of $\RB^{q}$ induces a (non-Hausdorff) manifold structure on $\Gamma_{q}$, which is thus an \'{e}tale diffeological groupoid.
\end{ex}

The following definition is a diffeological analogue of Haefliger's, in \cite[Section 2]{ha6}.

\begin{defn}\label{Gammacocycle}
  Let $X$ be a diffeological space and $\Gamma$ a diffeological groupoid.
  A \textbf{$\Gamma$-cocycle on $X$ over \( \UU \)} is a smooth morphism $h \colon \check{\UU}\rightarrow\Gamma$ of diffeological groupoids, with $\UU$ a D-open cover of $X$.
  Two $\Gamma$-cocycles over respective open covers $\UU$ and $\VV$ are \textbf{equivalent} if they are restrictions to $\check{\UU}$ and $\check{\VV}$ of a smooth $\Gamma$-cocycle over $\UU\cup\VV$.  An equivalence class of $\Gamma$-cocycles is called a \textbf{$\Gamma$-structure}.  Two $\Gamma$-structures are \textbf{concordant} if there exists a $\Gamma$-structure on $X\times[0,1]$ inducing the given structures on $X \cong X\times\{0\}$ and $X\cong X\times\{1\}$ respectively.
\end{defn}

Note that a Haefliger structure is a $\Gamma_{q}$-structure of codimension $q$.
Given a smooth map \( \eta \colon Y \to X \) of diffeological spaces, the pullback of \( h \) is the composition \( h \circ \eta \colon \eta^{*}\tilde{\UU} \to \UU \to \Gamma \), where \( \eta^{*}\UU \) is the pullback covering on \( Y \).

The next definition specifies the kinds of $\Gamma$-structures to which the classifying theorem will apply; it is a diffeological analogue of that found in \cite[Section 6]{ha6}.

\begin{defn}\label{numerability}
  For a diffeological space \( X \) with its D-open topology, a smooth, countable partition of unity $\{\lambda_{\alpha}\}_{\alpha\in\NB}$ is \textbf{locally finite} if the covering of \( X \) by supports $\overline{\lambda_{\alpha}^{-1}(0,1]}$ is locally finite, and \textbf{subordinate} to an open cover $\{U_{\alpha}\}_{\alpha\in\NB}$ of $X$ if $\overline{\lambda_{\alpha}^{-1}(0,1]}\subset U_{\alpha}$ for all $\alpha$.
  A countable open cover of $X$ is \textbf{smoothly numerable} if it admits a subordinate, locally finite, smooth partition of unity.
  A smooth $\Gamma$-structure on $X$ is \textbf{smoothly numerable} if it admits a representative cocycle over a countable, smoothly numerable open cover, and two smooth, numerable $\Gamma$-structures on $X$ are said to be \textbf{smoothly, numerably concordant} if there exists a smoothly numerable concordance between them.
\end{defn}

The following theorem is a generalisation of Haefliger's classifying theorem \cite[Theorem 7]{ha6} to the diffeological setting. Its proof is different from those of existing similar theorems in nontrivial ways, and can be found in Appendix \ref{app: smooth Haefliger classifying}.

\begin{thm}\label{classify}
  Let $\Gamma$ be a diffeological groupoid.
  There exists a diffeological space $B\Gamma$, equipped with a tautological smoothly numerable $\Gamma$-structure $\gamma$, for which the following hold:
	\begin{enumerate}
		\item For any smoothly numerable $\Gamma$-structure $h$ on a diffeological space $X$, there is a smooth map $\eta \colon X\rightarrow B\Gamma_{q}$ such that $h = \eta^{*}\gamma$.
		\item If $\eta_{0},\eta_{1} \colon X\rightarrow B\Gamma$ are smooth maps, then the $\Gamma$-structures $\eta_{0}^{*}\gamma$ and $\eta_{1}^{*}\gamma$ are smoothly, numerably concordant if and only if $\eta_{0}$ and $\eta_{1}$ are smoothly homotopic.
	\end{enumerate}
\end{thm}

\subsection{Universal characteristic map via Haefliger bundles}

In this section we define for any smoothly numerable Haefliger structure on a diffeological space the associated \emph{Haefliger bundle}, which is a natural, locally trivial principal $G^{k}_{q}$-bundle%
\footnote{In the diffeological setting, principal does not imply locally trivial, but only locally trivial along plots \cite[Section 8.13]{diffeology}.}
on the space.
As we will see in the next section, the Haefliger bundle of a regular Haefliger structure is isomorphic to the transverse frame bundle of the associated regular foliation.
Furthermore, while transverse frames are not well-defined over singularities, Haefliger bundles are defined globally even in the singular case.

Consider a smoothly numerable Haefliger cocycle $h \colon \check{\UU}\rightarrow\Gamma_{q}$ on a diffeological space $X$.  For $k\in\NB\cup\{\infty\}$, the \emph{Haefliger $k$-frame bundle of $h$}, or simply the \emph{$k$-Haefliger bundle}, is the smooth quotient
\begin{equation}\label{haefy1}
\Fr_{k}(h):=\bigg(\bigsqcup_{\alpha\in\NB}(s\circ h_{\alpha\alpha})^{*}\Fr_{k}(\RB^{q})\bigg)\big/\sim ,
\end{equation}
the equivalence relation $\sim$ given by $(x,\alpha,\varphi) \sim (x,\beta,h_{\beta\alpha}(x)\cdot\varphi)$ for all $x\in U_{\alpha}\cap U_{\beta}$.
Note moreover that by the $\Diff(\RB^{q})$-invariance of the tautological forms $\omega^{k}$ on $\Fr_{k}(\RB^{q})$, the pullbacks $(s\composed h_{\alpha\alpha})^{*}\omega^{k}$ glue to give an $\af_{q}^{k}$-valued 1-form $\omega^{k}_{h}$ on $\Fr_{k}(h)$.  We have the following result.

\begin{prop}\label{functorial}
	Let $X$ be a diffeological space and let $k\in\NB\cup\{\infty\}$.
	\begin{enumerate}
		\item If $\UU$ is an open cover of $X$ and $h \colon \check{\UU}\rightarrow\Gamma_{q}$ represents a smoothly numerable Haefliger structure on $X$, then $\Fr_{k}(h)$ is a smooth, principal $G^{k}_{q}$-bundle over $X$, locally trivial over $\UU$.
		\item If $\eta \colon X\rightarrow B\Gamma_{q}$ is a smooth map, then $\Fr_{k}(\eta^{*}\gamma)$ and $\eta^{*}\Fr_{k}(\gamma)$ are canonically isomorphic as principal $G^{k}_{q}$-bundles over $X$, and under this isomorphism $\eta^{*}\omega_{\gamma}^{k}$ identifies with $\omega_{\eta^{*}\gamma}^{k}$.
		\item If $h_{0}$ and $h_{1}$ are smoothly, numerably concordant Haefliger structures on $X$, then $\Fr_{k}(h_{0})$ and $\Fr_{k}(h_{1})$ are isomorphic as principal $G^{k}_{q}$-bundles over $X$.
	\end{enumerate}
	As a consequence, isomorphism classes of Haefliger bundles are functorial for diagrams
	\begin{center}
	\begin{tikzcd}[row sep={8mm,between origins}]
		X  \ar[dd] \ar[dr] & \\ & B\Gamma_{q} \\ Y \ar[ur] &
	\end{tikzcd}
	\end{center}
	of smooth maps that commute up to homotopy.
\end{prop}

\begin{proof}
	Item (1) follows from the triviality of $\Fr_{k}(\RB^{q})\rightarrow\RB^{q}$, while item (2) follows from the naturality of pullbacks.  The final item follows from \cite[Lemma 3.11]{magnot3}.  Note that the hypothesis there that $X$ be Hausdorff, second-countable and smoothly paracompact  is required only for the existence of a smoothly numerable open cover over which the bundle is locally trivial.  In our setting, this follows from the hypothesis that $h_{0}$ and $h_{1}$ are smoothly numerable Haefliger structures.
\end{proof}

The tautological form $\omega^{\infty}_{h}$ on the Haefliger bundle $\Fr_{\infty}(h)\rightarrow X$ induces a homomorphism
\[
A^{*}(\af_{q})_{\O(q)}\rightarrow\Omega^{*}(\Fr_{\infty}(h)/\O(q)).
\]
Thus, given a smooth section $X\rightarrow\Fr_{\infty}(h)/\O(q)$, one obtains a homomorphism $A^{*}(\af_{q})_{\O(q)}\rightarrow\Omega^{*}(X)$ of DGAs.  Our next theorem shows that, despite the potentially pathological topology of $X$, such sections are in abundance, and their induced maps on cohomology unambiguous.

\begin{thm}\label{sections}
	Let $h$ be a smoothly numerable Haefliger structure on a diffeological space $X$.  Then $\Fr_{\infty}(h)/\O(q)\rightarrow X$ admits smooth sections, and any two such sections are smoothly homotopic.
\end{thm}

\begin{proof}
  It suffices to prove the result for $X = B\Gamma_q$, by the the universal property of pullbacks.
  The theorem is then essentially a consequence of the fact that the fibre $G_{q}^{\infty}/\O(q)$ of the bundle is (smoothly) contractible.  When taken with the strong topology, $B\Gamma_{q}$ has the homotopy type of a CW complex \cite[Theorem 1.6]{ha3}, and the existence of \emph{continuous} sections is then classical.  An explicit construction of a tautological \emph{smooth} section when $B\Gamma_q$ is regarded as a diffeological space is given in Appendix \ref{app: sections of Haefliger bundles}.
\end{proof}

By the homotopy-invariance of de Rham cohomology, Theorem \ref{sections} enables the following definition of de Rham characteristic maps for smoothly numerable Haefliger structures.

\begin{defn}\label{universalmap}
	If $h$ is a smoothly numerable Haefliger structure on a diffeological space $X$, the \textbf{characteristic map} associated to $h$ is the homomorphism $H^{*}(A(\af_{q});\O(q))\rightarrow H^{*}(\Omega(X))$ induced by the composite
	\[
	A^{*}(\af_{q})_{\O(q)}\xrightarrow{\omega_{h}^{\infty}}\Omega^{*}(\Fr_{\infty}(h)/\O(q))\xrightarrow{\sigma^{*}}\Omega^{*}(X),
	\]
	where $\sigma \colon X\rightarrow\Fr_{\infty}(h)/\O(q)$ is any smooth section.  In particular, the characteristic map $u \colon H^{*}(A(\af_{q});\O(q))\rightarrow H^{*}(\Omega(B\Gamma_{q}))$ associated to the tautological Haefliger structure $\gamma$ on $B\Gamma_{q}$ is called the \textbf{universal characteristic map}.
\end{defn}

It follows easily from Proposition \ref{functorial} that characteristic maps for smoothly numerable Haefliger structures are functorial under smooth maps. In particular if $h = \eta^{*}\gamma$ is the smoothly numerable Haefliger structure defined by a smooth map $\eta \colon X\rightarrow B\Gamma_{q}$, then the characteristic map associated to $h$ is equal to 
\[ \eta^{*}\circ u \colon H^{*}(A(\af_{q});\O(q))\rightarrow H^{*}(\Omega(B\Gamma_{q})))\rightarrow H^{*}(\Omega(X)) . \]

The following is a corollary of Theorems \ref{sections} and \ref{woq}, and the fact that the map on \( WO_{q} \) factors through the \( 2 \)-frame bundle.
\begin{cor}\label{cor: WOq factors 2jets}
  Let $X$ be a diffeological space and $h = \eta^{*}\gamma$ a smoothly numerable Haefliger structure on $X$.
  If $\sigma \colon X\rightarrow\Fr_{2}(h)/\O(q)$ is a smooth section,
  the homomorphism of DGAs
  \[
    WO_{q}\xrightarrow{\omega^{2}_{h}}\Omega^{*}(\Fr_{2}(h)/\O(q))\xrightarrow{\sigma^{*}}\Omega^{*}(X)
  \]
  induces a map on cohomology equal to the characteristic map $\eta^{*}\circ u \colon H^{*}(WO_{q})\rightarrow H^{*}(\Omega(X))$.
\end{cor}
The advantage of this reduction is that sections of $\Fr_{2}(h)/\O(q)\rightarrow X$ can be induced by simple geometric structures (connections and curvatures).
In other words, Chern-Weil theory.

\section{Regular foliations and Haefliger bundles}\label{haefybundle}
This section uses the diffeological machinery of the previous section to provide an alternative proof of Bott's result that the Chern-Weil map for a regular foliation recovers the characteristic classes,
which proof will readily generalise to the Haefliger-Singular case in the following sections.
First are two propositions.

Fix \( (M, \FF) \) a regular foliation of codimension \( q \), and recall that for $1\leq k\leq\infty$ we denote by $\Fr_k(M/\FF)$ the transverse $k$-frame bundle (Definition \ref{framebundle}) and by $\Fr_k(h_{\FF})$ the Haefliger bundle as defined in Equation \eqref{haefy1}.
The following Proposition identifies these bundles.
\begin{prop}\label{framehaefy}
	Let $(M,\FF)$ be a regular foliation of codimension $q$.  For all $1\leq k\leq\infty$, there are canonical isomorphisms $i_{k} \colon \Fr_{k}(M/\FF) \to \Fr_{k}(h_{\FF})$ of principal $G^{k}_{q}$-bundles.
  These commute with the natural projection maps.
\end{prop}

\begin{proof}
	Assume $(M,\FF)$ to be associated to Haefliger charts $(U_{\alpha},f_{\alpha})$, with $f_{\alpha}$ submersive. The local maps $i_{k,\alpha} \colon \Fr_{k}(M/\FF)|_{U_{\alpha}}\rightarrow \Fr_{k}(h_{\FF})|_{U_{\alpha}}$ defined by
	\[
	i_{k,\alpha}\big(j^{k}_{\vec{0},\t}(u)\big)\mapsto
	\bigl[(u(0), \alpha, j^{k}_{\vec{0}}(f_{\alpha}\circ u))\bigr]_{\sim}
	\]
	glue to give the isomorphism $i_{k}$.
	Indeed, since each $f_{\alpha}$ is a submersion and each \( u \) a frame, the composition $f_{\alpha}\circ u$ is a local diffeomorphism of $\RB^{q}$.
	The maps \( i_{k,\alpha} \) are manifestly equivariant for the respective right actions of \( G^{k}_{q} \), and
	on each overlap $U_{\alpha}\cap U_{\beta}$, one has
	\[
		\bigl(x, \alpha, j^{k}_{\vec{0}}(f_{\alpha}\circ u)\bigr) =
	 	\bigl(x,\alpha,j^{k}_{\vec{0}}(h^{x}_{\alpha\beta}\circ f_{\beta}\circ u)\bigr)
		\sim \bigl(x,\beta,j^{k}_{\vec{0}}(f_{\beta}\circ u)\bigr)
	\]
	for each $j^{k}_{\vec{0},\t}(u) \in \Fr_{k}(M/\FF)|_{U_{\alpha}\cap U_{\beta}}$, so that \( i_{k,\alpha}(x,j^{k}_{\vec{0},\t}(u)) = i_{k,\beta}(x,j^{k}_{\vec{0},\t}(u)) \).
\end{proof}

The isomorphisms $i_{k} \colon \Fr_{k}(M/\FF)\rightarrow\Fr_{k}(h_{\FF})$ pull back the tautological forms $\omega^k_h$ from $\Fr_{k}(h_{\FF})$, in particular yielding the \( \RB^{q} \)- and \( \glf(\RB^{q}) \)-valued forms
\begin{align*}
  \omega^{i} := i_{1}^{*}(\omega^{1}_{h})^i \\
  \omega^{i}_{j} := i_{2}^{*}(\omega^{2}_{h})^i_j
\end{align*}
on $\Fr_2(M/\FF)$. These furthermore satisfy the structure equations
\[ \d\omega^{i} = -\omega^{i}_{j} \w \omega^{j} . \]

The next Proposition is folklore, although elements of its proof appear in a less precise form in \cite[Section 5.2]{macr1}.
We give the proof here for completeness.

\begin{prop}\label{correspondence}
  Let $(M,\FF)$ be a regular foliation of codimension $q$.
  Torsion-free Bott connections on $\nuFF$ are in bijective correspondence with $\GL(\RB^{q})$-equivariant sections of $\Fr_{2}(M/\FF)\rightarrow\Fr_{1}(M/\FF)$.
\end{prop}
\begin{proof}
  \newcommand{\downomega}{\check{\omega}}

  Consider an equivariant section \( \sigma_{2} \colon \Fr_{1}(M/\FF) \to \Fr_{2}(M/\FF) \).
  Choose an arbitrary local section \( \sigma_{1} \) of \( \Fr_{1}(M/\FF) \to M \), and use the sections \( \sigma_{1} \) and \( \sigma_{2} \circ \sigma_{1} \) to pull the tautological forms \( \omega^{i}, \omega^{i}_{j} \) to forms \( \downomega^{i}, \downomega^{i}_{j} \) on \( M \).
  (Here we may work locally, so without loss suppose that the domain of \( \sigma_{1} \) is all of \( M \).)
  The \( \downomega^{i} \) annihilate \( \tFF \) by definition, and so descend to a well-defined coframing of the normal bundle \( \nuFF = TM / \tFF \).
  Let \( e_{i} \) denote the framing of \( \nuFF \) dual to the \( \omega^{i} \).
  We may define a connection relative this framing by
  \[ \nabla_{X} e_{i} = \downomega^{j}_{i}(X_{\nu}) e_{j} + [X_{\FF}, e_{i}]_{\nu} \qquad \mbox{ for } X \in \XF(M) . \]
  (Recall the notation for a splitting of \( X = X_{\nu} + X_{\FF} \) into components normal and tangential to \( \tFF \).)
  By construction this is a Bott connection.
  It is independent of the choice of coframing \( \sigma_{1} \), because any two coframings differ by the action of a \( \GL(\RB^{q}) \)-valued function and the assumption that \( \downomega^{i}_{j} \) depends equivariantly on this choice.
  The connection is torsion-free by the following calculation:
  \[ \downomega^{i}([e_{j}, e_{k}]) = - \d \downomega^{i}(e_{j}, e_{k}) = \downomega^{i}_{l} \w \downomega^{l}(e_{j}, e_{k}) = \downomega^{i}(\nabla_{e_{j}} e_{k} - \nabla_{e_{k}} e_{j}) . \]

Conversely, let $\nabla^{\nu}$ be a torsion-free Bott connection on $\nuFF$.
The aim is to construct from \( \nabla^{\nu} \) an equivariant section $\sigma \colon \Fr_{1}(M/\FF) \rightarrow \Fr_{2}(M/\FF)$.
To this end, fix an arbitrary connection \( \nabla^{\FF} \) on \( \tFF \), and consider the affine connection 
\[ \nabla_{X} Y = \nabla^{\nu}_{X} Y_{\nu} + \nabla^{\FF}_{X} Y_{\FF} . \]
Define a section $\sigma \colon \Fr_{1}(M/\FF)\rightarrow\Fr_{2}(M/\FF)$ by taking the \( 2 \)-jet of the exponential map;
for a frame \( \varepsilon = (e_{i}) \in \Fr_{1}(M/\FF) \) of \( \nuFF \) over a point \( p \in M \), there is a canonical isomorphism \( \RB^{q} \cong \nuFF_{p} \), and the exponential map defines a smooth map
\[ \begin{tikzcd}[row sep={8mm,between origins}]
u_{\varepsilon} \colon \RB^{q} \cong \nuFF_{p} \ar[r] & M \\
v = v^{i}e_{i} \ar[r, mapsto] & \exp_{\nabla}(v) .
\end{tikzcd} \]
Define \( \sigma \) by \( \sigma(\varepsilon) = j^{2}_{\vec{0},\t}(u_{\varepsilon}) \), a choice which is manifestly equivariant.

While the function \( u \) depends on the choice of \( \nabla^{\FF} \), its transverse \( 2 \)-jet does not.
To see this, extend the frame \( e_{i} \) of \( \nuFF_{p} \) to a frame \( e_{i}, f_{a} \) of \( T_{p} M \).
Relative this frame, the Christoffel symbols are given by
\begin{align*}
\nabla_{e_{i}} e_{j} = \Gamma_{ij}^{k} e_{k} + \Gamma_{ij}^{a} f_{a}, \qquad
\begin{array}{c}
\nabla_{e_{j}} f_{a} = \Gamma_{ja}^{k} e_{k} + \Gamma_{ja}^{b} f_{b} \\
\nabla_{f_{a}} e_{j} = \Gamma_{aj}^{k} e_{k} + \Gamma_{aj}^{b} f_{b} ,
\end{array}  \qquad
\nabla_{f_{a}} f_{b} = \Gamma_{ab}^{k} e_{k} + \Gamma_{ab}^{c} f_{c} .
\end{align*}
The exponential map is determined by geodesics;
a path $\gamma(t)$ through \( p \in M \), with tangential components \( \dot\gamma = (\gamma^{i}, \gamma^{a}) \),
is geodesic if and only if it satisfies the equation
\[ \ddot\gamma^{A}(t) = -\Gamma^{A}_{BC}(\gamma(t))\dot\gamma^{B}(t)\dot\gamma^{C}(t) \qquad A, B, C = 1,\ldots n \]
for \( t \) in the domain of \( \gamma \).
But for any geodesic through \( p \) and tangent to \( \nuFF \) at \( p \), this reduces to $\ddot\gamma^{A}(0) = -\Gamma^{A}_{jk}(x)\dot\gamma^{j}(0)\dot\gamma^{k}(0)$, and the normal component of this reduces further to
\[
\ddot\gamma^{i}(0) = -\Gamma^{i}_{jk}(x)\dot\gamma^{j}(0)\dot\gamma^{k}(0).
\]
Here only `normal' indices \( i,j,k \) appear, and these particular components of the Christoffel symbol depend only on the Bott connection, so, the normal component of the two-jet of the exponential map is independent of the choice of \( \nabla^{\FF} \).
\end{proof}

With Proposition \ref{correspondence} in hand, we can reformulate the Chern-Weil theorem for regular foliations, cf. Theorem \ref{guegorlet}.
Recall that if $\epsilon$ is a Euclidean structure on a vector bundle, then any connection $\nabla$ on the bundle induces an $\epsilon$-compatible connection $\nabla_{\epsilon}$.
The connection \( \nabla_{\epsilon} \) can be defined relative an orthonormal framing by keeping the antisymmetric component of the connection form of \( \nabla \), which choice is compatible for two orthonormal coframings because they differ by the action of a function \( M \to \O(n) \).
\begin{thm}[Chern-Weil theorem for regular foliations]\label{cwreg}
  Let $(M,\FF)$ be a regular foliation of codimension $q$, and \( \eta \colon M \to B \Gamma_{q} \) a classifying map for \( \FF \).
  Suppose \( \nuFF \) is equipped with a Euclidean structure $\epsilon$ and a torsion-free Bott connection $\nabla$.
  Denote by $\nabla_{\epsilon}$ the $\epsilon$-compatible connection on \( \nuFF \) induced by $\nabla$, and for $t\in[0,1]$ denote by $\nabla^{t}:=t\nabla+(1-t)\nabla_{\epsilon}$ the affine combination on \( M \times I \).
  The map of algebras $\lambda_{\epsilon,\nabla} \colon WO_{q}\rightarrow\Omega^{*}(M)$ defined on generators of \( WO_{q} \) by
  \[
    \lambda_{\epsilon,\nabla}(c_{i}):=\Tr(R^{\w i}_{\nabla})\in\Omega^{2i}(M)
  \]
  \[
    \lambda_{\epsilon,\nabla}(h_{i}) := i \int_{0}^{1}\Tr\left((\nabla-\nabla_{\epsilon})\w R^{\w(i-1)}_{\nabla^{t}}\right)\,dt\in\Omega^{2i-1}(M)
  \]
  is a DGA map.
  Descending to cohomology, the diagram
  \begin{equation}\label{eq:WOq CHW cohomology}
    \begin{tikzcd}
      H^{*}(WO_{q}) \ar[r,"u"] \ar[dr,"\lambda_{\epsilon,\nabla}"'] & H^{*}(\Omega(B\Gamma_{q})) \ar[d,"\eta^{*}"] \\
      & H^{*}(\Omega(M))
    \end{tikzcd}
  \end{equation}
  commutes.
\end{thm}
\begin{proof}
  Let $\sigma_{\epsilon} \colon M\rightarrow\Fr_{1}(M/\FF)/\O(q)$ be the section defined by the Euclidean structure, locally determined by any choice of orthonormal framing.
  Per Proposition \ref{correspondence}, the Bott connection \( \nabla \) determines a section $\sigma_{\nabla} \colon \Fr_{1}(M/\FF)\rightarrow\Fr_{2}(M/\FF)$, which is equivariant, and thus descends to a section \( \sigma_{\nabla} \colon \Fr_{1}(M/\FF) / \O(q) \rightarrow\Fr_{2}(M/\FF) / \O(q) \).
  It is a direct calculation that for each generator \( a = c_{i} \) or \( h_{i} \) of \( WO_{q} \), the image $\lambda_{\epsilon,\nabla}(a)$ is the pullback by $\sigma_{\nabla} \sigma_{\epsilon}$ of $a(i_{2}^{*}\omega^{2})$, as in the left triangle of the following diagram.
  \begin{equation}\label{eq:WOq CHW chains}
  \begin{tikzcd}
      WO_{q} \ar[rr, bend left=8mm, "\omega^{2}"] \ar[r, "i_{2}^{*}\omega^{2}" pos=.7] \ar[dr, "\lambda_{\epsilon,\nabla}" swap] & \Omega^{*}(\Fr_{2}(M/\FF)/\O(q)) \ar[d, "(\sigma_{\nabla}\sigma_{\epsilon})^{*}"] & \Omega^{*}(Fr_{2}(\gamma)/\O(q)) \ar[l] \ar[d, "\sigma^{*}"] \\
      & \Omega^{*}(M) & \Omega^{*}(B\Gamma_{q}) \ar[l, "\eta^{*}"]
    \end{tikzcd}
  \end{equation}
  Up to the isomorphism \( i_{2}^{*} \), the assignment \( a \mapsto a(i_{2}^{*}\omega^{2}) \) is the same as the map \( \omega^{2} \) in Corollary \ref{cor: WOq factors 2jets}, and so by that Corollary, the map on generators extends to a DGA map \( \lambda_{\epsilon, \nabla} \).
  The commutativity of diagram \eqref{eq:WOq CHW cohomology} then follows from the commutativity up to homotopy of diagram \eqref{eq:WOq CHW chains}, recalling that the universal characteristic map \( u \) is determined by the map of forms \( \sigma^{*} \circ \omega^{2} \).
\end{proof}

\section{Chern-Weil for Haefliger-singular foliations}\label{singularCHW}

\subsection{Adapted geometries and the Chern-Weil map}

Fix a codimension-\( q \) Haefliger-singular foliation \( (M, \FF) \), with smooth classifying map $\eta_{\FF} \colon M\rightarrow B\Gamma_{q}$ and accompanying Haefliger cocycle $h_{\FF} := \eta_{\FF}^{*}\gamma$.
Recall that the regular set \( \tilde{M} \) is dense.
Over the regular set, a metric \( \epsilon \) on \( \nuFF \) induces a section $\sigma_{\epsilon} \colon \tilde{M}\rightarrow\Fr_{1}(\tilde{M}/\FF)/\O(q)$ and a Bott connection \( \nabla \) induces a section $\sigma_{\nabla} \colon \Fr_{1}(\tilde{M}/\FF)/\O(q)\rightarrow\Fr_{2}(\tilde{M}/\FF)/\O(q)$.
The isomorphism \( i_{2} \colon \Fr_{2}(\tilde{M}/\FF)/\O(q) \hookrightarrow \Fr_{2}(h_{\FF})/\O(q) \) of Proposition \ref{framehaefy} is also defined over \( \tilde{M} \).
\begin{defn}
  An \textbf{adapted geometry} for $(M,\FF)$ is a pair $(\epsilon,\nabla)$ of a Euclidean metric $\epsilon$ and a torsion-free Bott connection $\nabla$ on \( \nuFF \) over \( \tilde{M} \) for which the composition $i_{2}\sigma_{\nabla}\sigma_{\epsilon} \colon \tilde{M}\rightarrow\Fr_{2}(h_{\FF})/\O(q)$
  extends smoothly to a section \( \sigma \) of \( \Fr_{2}(h_{\FF})/\O(q) \to M \).
\end{defn}
One may worry that such special geometries don't exist;
in fact, every Haefliger-singular foliation admits adapted geometries, which is a consequence of Corollary \ref{existence}.
Before turning to existence, we describe how an adapted geometry allows for the construction of a Chern-Weil map on Haefliger-singular foliations.
\begin{thm}[Chern-Weil for Haefliger-singular foliations]\label{cwsing}
  Let $(M,\FF)$ be a codimension-\( q \) Haefliger-singular foliation with classifying map \( \eta_{\FF} \colon M \to B\Gamma_{q} \), and let $(\epsilon, \nabla)$ be an adapted geometry for \( \FF \).
  Denote by $\nabla_{\epsilon}$ the $\epsilon$-compatible connection on $\nuFF$ induced by $\nabla$, and by $\nabla^{t}:=t\nabla+ (1-t)\nabla_{\epsilon}$ the affine combination on \( M \times I \).
  The forms
  \[
    \lambda_{\epsilon,\nabla}(c_{i}):=\Tr(R_{\nabla}^{\w i})\in\Omega^{2i}(\tilde{M})
  \]
  \[	\lambda_{\epsilon,\nabla}(h_{i}):=i\int_{0}^{1}\Tr\big((\nabla-\nabla_{\epsilon})\w R_{\nabla_{t}}^{\w(i-1)}\big)\,dt\in\Omega^{2i-1}(\tilde{M})
  \]
  extend to globally-defined smooth forms on $M$.
  Furthermore, the resulting homomorphism $\lambda_{\epsilon,\nabla} \colon WO_{q}\rightarrow\Omega^{*}(M)$ of DGAs makes the following diagram commute.
  \begin{center}
    \begin{tikzcd}
      H^{*}(WO_{q}) \ar[r,"u"] \ar[dr,"\lambda_{\epsilon,\nabla}"'] & H^{*}(\Omega(B\Gamma_{q})) \ar[d,"\eta_{\FF}^{*}"] \\ & H^{*}(\Omega(M))
    \end{tikzcd}
  \end{center}
\end{thm}

\begin{proof}
  The proof is nearly identical to the regular case, Theorem \ref{cwreg}.
  As in the regular case, the image by \( \lambda_{\epsilon, \nabla} \) of any generator \( a = c_{i} \) or \( h_{i} \) is the pullback of \( a(i_{2}^{*} \omega^{2}) \) by \( \sigma_{\nabla} \sigma_{\epsilon} \) (commutativity of the bottom left triangle in the following diagram).
  \begin{equation*}
  \begin{tikzcd}
    WO_{q} \ar[ddr, "\lambda_{\epsilon,\nabla}" swap, bend right] \ar[dr, "i_{2}^{*}\omega^{2}" swap] \ar[drr, "\omega_{2}" swap, bend left=10] \ar[drrr, "\omega_{2}" swap, bend left=10] & & & \\
    & \Omega^{*}\left(\frac{\Fr_{2}(M/\FF)}{\O(q)}\right) \ar[d, "(\sigma_{\nabla}\sigma_{\epsilon})^{*}" swap] & \Omega^{*}\left(\frac{\Fr_{2}(h_{\FF}|_{\tilde{M}})}{\O(q)}\right) \ar[l, "i_{2}^{*}" swap] \ar[dl] & \Omega^{*}\left(\frac{\Fr_{2}(h_{\FF})}{\O(q)}\right) \ar[l] \ar[d, "\sigma^{*}"] \\
     & \Omega^{*}(\tilde{M}) & & \Omega^{*}(M) \ar[ll]
    \end{tikzcd}
  \end{equation*}
  But the form \( a(\omega^{2}) \) is defined on all of \( \Fr_{2}(h_{\FF})/\O(q) \) over \( M \), and its pullback by \( \sigma \) is the desired extension of \( \lambda_{\epsilon, \nabla} \), where \( \sigma \colon M \to \Fr_{2}(\Fr_{2}(h_{\FF}/O(q))) \) the smooth section extending \( i_{2}\sigma_{\nabla} \sigma_{\epsilon} \).
  
  Commutativity of the homology diagram follows by the same argument as in Theorem \ref{cwreg}.
\end{proof}

\subsection{Functoriality}
Given a regular foliation \( (N, \FF) \) and a map \( f \colon M \to N \) transverse to \( \FF \), the pullback foliation \( f^{*}\FF \) is again a regular foliation.
Furthermore, the regular Chern-Weil map pulls back the characteristic classes in a functorial manner.
The situation is similar for Haefliger-singular foliations and Haefliger-singular maps, which we turn to now.
\begin{defn}\label{gentrans}
  Let $(N,\FF')$ be a Haefliger-singular foliation.
  A smooth map $f \colon M\rightarrow N$ is \textbf{regular at $x\in M$} if for some (and hence any) Haefliger chart \( h_{\alpha} \colon U_{\alpha} \to \RB^q \) with \( f(x) \in U_{\alpha} \), it holds that \( h_{\alpha} \circ f \) is a submersion at \( x \).
  The map $f$ is \textbf{Haefliger-singular} if it is regular on a dense open subset of \( M \).
\end{defn}

It follows immediately from Definition \ref{gentrans} that the pullback of a Haefliger-singular foliation by a Haefliger-singular map is again Haefliger-singular.
Similarly, one has functoriality of adapted geometries.
\begin{thm}\label{functoriality}
  Let $(N,\FF')$ be a codimension-\( q \) Haefliger-singular foliation, $f \colon M\rightarrow (N,\FF')$ a Haefliger-singular map, and $\FF = f^{*}\FF'$ the induced Haefliger-singular foliation on $M$.
  If $(\epsilon', \nabla')$ is an adapted geometry for $(N,\FF')$, then $(\epsilon, \nabla) = (f^{*}\epsilon', f^{*}\nabla')$ is an adapted geometry for $(M, \FF)$.
  Furthermore, the following diagram commutes.
  \begin{equation}\label{functdiagram1}
    \begin{tikzcd}[row sep={8mm,between origins}]
      & \Omega^{*}(N) \ar[dd,"f^{*}"] \\
      WO_{q} \ar[ur,"\lambda_{\epsilon',\nabla'}"] \ar[dr, "\lambda_{\epsilon,\nabla}"']  & \\
      & \Omega^{*}(M) 
    \end{tikzcd} 
  \end{equation}
\end{thm}
\begin{proof}
  Consider the diagram, with \( \sigma \) the unique smooth section extending \( i_{2}\sigma_{\nabla'}\sigma_{\epsilon'} \).
  \begin{equation}\label{functdiagram2}
    \begin{tikzcd}
      \tilde{N} \ar[rrr, hook] \ar[dr, "\sigma_{\nabla'}\sigma_{\epsilon'}"] & & & N \ar[dl, "\sigma"] \\
      & \Fr_{2}(\tilde{N}/\FF')/\O(q) \ar[r,hook,"i_{2}"] & \Fr_{2}(h_{\FF'})/\O(q) & \\
      & \Fr_{2}(\tilde{M}/\FF)/\O(q) \ar[r,hook,"i_{2}"] \ar[u] & \Fr_{2}(h_{\FF})/\O(q) \ar[u] & \\
      \tilde{M} \ar[uuu, "f"] \ar[rrr, hook] \ar[ur,"\sigma_{\nabla}\sigma_{\epsilon}"'] & & & M \ar[uuu, "f"] \ar[ul, "f^{*}\sigma" swap, dashed]
    \end{tikzcd}
  \end{equation}
  The left hand square commutes because the pullback data $\epsilon$ and $\nabla$ are associated to the pullbacks by $f$ of the corresponding sections $\sigma_{\epsilon'}$ and $\sigma_{\nabla'}$.
  The right hand square commutes by construction if we take the smooth pullback section \( f^*\sigma \), after noting that functoriality of the Haefliger bundle affords the identification \( \Fr_{2}(h_{\FF})\cong f^{*}\Fr_{2}(h_{\FF'}) \).
  The top two middle squares commute by definition.

  At any regular point \( x \in \tilde{M} \), one sees that \( i_{2}\sigma_{\nabla}\sigma_{\epsilon}(x) = f^{*}\sigma(x) \) by chasing around the top of the diagram, noting that the fibre of \( \Fr_{2}(h_{\FF})/O(q) \) over \( x \) is canonically isomorphic to the fibre of \( \Fr_{2}(h_{\FF'})/O(q) \) over \( f(x) \).
  As \( \tilde{M} \) is dense, the section \( f^{*}\sigma \) is the unique extension of \( i_{2} \sigma_{\nabla}\sigma_{\epsilon} \), and so, having constructed a smooth extension, the geometry $(\epsilon,\nabla)$ is seen to be adapted.

  The commutativity of Diagram \eqref{functdiagram1} follows also from commutativity of Diagram \eqref{functdiagram2}, because the Chern-Weil morphism factors through the forms on the Haefliger bundle.
\end{proof}

Theorem \ref{functoriality} can be used to quickly demonstrate the following.

\begin{cor}\label{existence}
  Every Haefliger-singular foliation admits adapted geometries.
\end{cor}
\begin{proof}
  Every geometry $(\epsilon,\nabla)$ on a regular foliation $(N,\FF')$ is adapted, so every Haefliger-singular map $f \colon (M, \FF) \rightarrow (N,\FF')$ into a regular foliation induces an adapted geometry for \( \FF \).
  So, it suffices to show that every Haefliger-singular foliation admits such a map into a regular foliation.
  This follows from Haefliger's graph construction, \cite[p. 188]{ha3}, which we recall for completeness.

  Fix a representative Haefliger cocycle \( h \) for \( (M, \FF) \) over some locally finite open cover $\{U_{\alpha}\}_{\alpha\in\NB}$ of \( M \), with Haefliger charts $f_{\alpha} \colon U_{\alpha}\rightarrow\RB^{q}$.
  Form an open manifold $G(h)$ by gluing appropriate neighbourhoods \( V_{\alpha} \) of \( \Gamma(f_{\alpha}) = \{ (x, f_{\alpha}(x)) \colon x \in U_{\alpha} \} \) in \( U_{\alpha} \times \RB^{q} \)
  along the change of chart maps associated to the cocycle.
  One obtains a regular foliation on \( G(h) \) by gluing the foliations given on each \( V_{\alpha} \) by the level sets of \( V_{\alpha} \hookrightarrow U_{\alpha} \times \RB^{q} \to \RB^{q} \).
  The \emph{graph} of \( h \) is the map \( i \colon M \to G(h) \) obtained by gluing the graphs \( U_{\alpha} \to \Gamma(f_{\alpha}) \), well-defined by construction of \( G(h) \).
\end{proof}

While the adapted geometries just described come from the graph construction, we envision in applications other constructions of adapted geometries, which could be more direct, depending on the calculation.
For one example, the extension of Godbillon-Vey forms to the singular case, described in a following subsection, does not rely on the graph of the Haefliger structure, but rather on an arbitrary choice of adapted geometry for the foliation.

\subsection{Chern-Weil for general Haefliger structures}

This section is devoted to proving the general applicability of our theory;
specifically, we show that the class of Haefliger-singular foliations with adapted geometries is homotopically identical to that of all smooth Haefliger structures on manifolds.
Let $\FS^{q}_{sing}$ denote the category whose objects are Haefliger-singular foliations of codimension $q$ with an adapted geometry, and whose morphisms are Haefliger-singular maps pulling back the foliation and geometry on the codomain to that on the domain.
Note that the objects in \( \FS^{q}_{sing} \) are necessarily on smooth manifolds of dimension at least \( q \).
Let $\HH^{q}_{man}$ denote the category whose objects are arbitrary smooth Haefliger structures on manifolds of dimension at least $q$, and whose morphisms are smooth functions pulling back the Haefliger structure on the codomain to that on the domain.

\begin{thm}
  The inclusion of categories $\FS_{sing}^{q}\hookrightarrow\HH^{q}_{man}$ is a homotopy equivalence.
  Consequently, the characteristic map for all smooth Haefliger structures on manifolds dimension at least \( q \) is given functorially by the Chern-Weil map for Haefliger-singular foliations.
\end{thm}

This Theorem is an immediate consequence of the following Proposition.

\begin{prop}\label{homotopy}
Given a codimension-\( q \) Haefliger-singular foliation \( (N,\FF') \) and a manifold \( M \) of dimension at least \( q \), every smooth map \( f \colon M \to N \) is homotopic to a Haefliger-singular map.
\end{prop}
\begin{proof}
  \newcommand{\V}{\mathcal{V}}
  We will show the stronger statement that the subset of Haefliger-singular maps is dense (in fact, residual) in \( \cinfinity{M,N} \) equipped with the strong topology.
  Since \( \cinfinity{M,N} \) is locally path connected in this topology, there exists a continuous path to some Haefliger-singular map, which induces the homotopy.
  Fix \( m = \dim M \) and \( n = \dim N \).

  First suppose that the foliation on \( N \) is regular.
  Associated to \( (N,\FF') \) is a regular tangent plane distribution \( D \subset TN \), of codimension \( q \).
  The argument proceeds through the Thom jet-transversality theorem, to show that the set of smooth maps \( M \to N \) transverse to \( D \) on a dense subset of \( M \) is itself dense in \( \cinfinity{M,N} \).
  (As an aside, the integrability of \( D \) is not used in the following argument, which works for any regular smooth distribution.)

  In the \( 1 \)-jet bundle \( J^{1}(M,N) \), let \( \V \) denote the subset of \( J^{1}(M,N) \) comprising \( 1 \)-jets that are not transverse to \( D \);
  Although \( \V \) is not generally a manifold, it is stratified as a finite union of manifolds, each of positive codimension.
  To see this, note that the bundle map \( J^{1}(M,N) \to M \times N \) restricts to a bundle map \( \V \to M \times N \), with fibre over each \( (x,y) \in M \times N \) a certain variety \( \V_{0} \), and \( \V_{0} \) is a finite union of manifolds, each of positive codimension in the fibre \( J^{1}(M,N)_{(x,y)} \cong \hom(\RB^{m}, \RB^{n}) \).
  More precisely, \( \V_{0} \) is, up to isomorphism, the variety of elements in \( \hom(\RB^{m}, \RB^{n}) \) that pull back to \( 0 \) the \( q \)-form \( \d y^{1} \w \cdots \w \d y^{q} \) of \( \RB^{n} \);
  this condition determines polynomial equations in the matrix coefficients, which can be written out explicitly if required.
  Cf. the proof of Theorem 3.2.6 in \cite{Hirsch-differentialTopology}.
  Enumerate this finite collection of manifolds as \( W_{a} \), so that \( \V = \cup W_{a} \).

  The decomposition of \( \V \) is enough to show that the set of Haefliger-singular maps is dense in \( \cinfinity{M, N} \).
  First, observe that for any smooth map \( g \colon M \to N \), with  \( 1 \)-jet lift \( j^{1}(g) \colon M \to J^{1}(M, N) \), the set of points where \( g \) is not transverse to \( D \) is precisely \( j^{1}(g)^{-1}(\V) = \cup j^{1}(g)^{-1}(W_{a}) \).
  But each \( W_{a} \) has positive codimension, so if \( j^{1}(g) \) is transverse to each \( W_{a} \), the set \( j^{1}(g)^{-1}(\V) \) is contained in a finite union of positive-codimension submanifolds, so has dense complement---such \( g \) is Haefliger-singular.
  On the other hand, by Theorem 3.2.8 of \cite{Hirsch-differentialTopology}, the set of elements in \( \cinfinity{M,N} \) that have \( 1 \)-jet lift transverse to \( W_{a} \) is a residual set in \( \cinfinity{M,N} \).
  Since \( \cinfinity{M,N} \) is Baire (\cite{Hirsch-differentialTopology} Theorem 2.4.2), residual sets are dense.

  Now suppose more generally that \( (N,\FF') \) is Haefliger-singular, so that the regular set \( \tilde{N} \) is a dense open submanifold of \( N \).
  Let \( D \) be the (regular) tangent distribution over \( \tilde{N} \), and \( \V \) the set of \( 1 \)-jets with target in \( \tilde{N} \) and not transverse to \( D \).
  As in the regular case, \( \V \) is a \( \V_{0} \) bundle over \( \tilde{N} \), a finite union of positive-codimension submanifolds in \( J^{1}(M,N) \).
  Theorem 3.2.8 of \cite{Hirsch-differentialTopology} still applies, and the set of maps densely transverse to \( D \) remains residual in \( \cinfinity{M,N} \).
  On the other hand, the following holds, with proof deferred momentarily.
  \begin{lemma}\label{thm-lem dense pullback-ers are dense}
    Given a dense open set \( \tilde{N} \subset N \), let \( A \) be the set of smooth maps \( g \colon M \to N \) such that \( g^{-1}(\tilde{N}) \) is dense in \( M \).
    The set \( A \) is residual in \( \cinfinity{M,N} \).
  \end{lemma}
  The singular set of a map \( g \colon M \to N \) is contained in \( g^{-1}(\Sigma') \cup j^{1}(g)^{-1}(\V) \), with \( \Sigma' \) the singular set of \( N \).
  As such, any map in the intersection of the set of maps \( g \colon M \to N \) such that \( g^{-1}(\Sigma') \) has dense complement, and the set of maps \( g \) so that \( j^{1}(g) \) is transverse to \( \V \), is a Haefliger-singular map.
  Both sets are residual, so their intersection is residual, so is dense, as required.
\end{proof}

\begin{proof}[Proof of Lemma \ref{thm-lem dense pullback-ers are dense}]
  Consider the evaluation map \( \ev \colon \cinfinity{M,N} \times M \to N \).
  Since \( \ev \) is continuous, open, and surjective, the pullback \( U = \ev^{-1}(\tilde{N}) \) is open and dense in \( \cinfinity{M,N} \times M \).
  (The evaluation map is open because it is open on each product open \( \UU \times V \), which is a union of \( \UU \times \{x\} \) over \( x \) in \( V \).
  Each \( \ev(\UU \times \{x\}) \) is already open, sufficing to take \( \UU \) basic in \( \cinfinity{M,N} \), say the set of all maps \( M \to N \) whose graph lies in some open neighbourhood \( W \) of the graph of a fixed \( g_{0} \colon M \to N \).
  The projection of the set \( W \cap (\{x\} \times N) \) is open in \( N \), and equals \( \ev(\UU \times \{x\}) \).)

  That the set \( A \) is residual is a consequence of the following more general claim.
  Given topological spaces \( X, Y \) and a subset \( U \subset X \times Y \), denote by
  \[ U_{x} := \pi_{Y}\bigl((\{x\}\times Y)\cap U\bigr)\subset Y \]
  the \( x \)-slice of \( U \) in \( Y \), and by \( \XX \) the set of points \( x \in X \) so that \( U_{x} \) is dense in \( Y \).
  The claim is that if \( Y \) is a separable metric space and \( U \) is open and dense in \( X \times Y \), then \( \XX \) contains a residual subset of \( X \).

  To show this, let \( y_{n} \) be a countable dense sequence in \( Y \), for each \( n \) let \( B_{n} \) be the ball of radius \( 1/n \) about \( y_{n} \), and let
  \[ U^{B_{n}} := \{ x \in X \colon (\{x\} \times B_{n}) \cap U \neq \varnothing \} = \{ x \in X \colon  B_{n} \cap U_{x} \neq \varnothing \} . \]
  Each \( U^{B_{n}} \) is dense in \( X \), for if not, there would be a nonempty open \( V \subset X \) so that \( V \cap U^{B_{n}} = \varnothing \), and then \( V \times B_{n} \) would be an open nonempty set disjoint from \( U \).
  The sets \( U^{B_{n}} \) are open because \( U \) is open.
  The (residual) countable intersection
  \[ \XX' = \bigcap_{n \in \NB} U^{B_{n}} \]
  is contained in \( \XX \) by construction, which is to say that for each \( x \in \XX' \), we have \( U_{x} \) dense in \( Y \).
  Indeed, fix \( x \in \XX' \), and for any \( y \in Y \), take a subsequence so that \( y_{n} \to y \).
  There is for each \( n \) some \( y'_{n} \in B_{n} \cap U_{x} \), so that \( (x,y'_{n}) \in U \).
  The new sequence \( y'_{n} \) also converges to \( y \), so \( U_{x} \) contains a sequence converging to \( y \) for all \( y \in Y \), i.e. \( U_{x} \) is dense in \( Y \).
\end{proof}

\subsection{The singular Godbillon-Vey algorithm}

The classical Godbillon-Vey algorithm \cite{gv} for a regular, transversely orientable foliation $(M,\FF)$ of codimension 1 proceeds by choosing a nowhere vanishing 1-form $\omega$ defining $\FF$ and a 1-form $\eta$ such that
\[
  \d\omega = \eta\wedge\omega .
\]
The $3$-form $\eta\wedge \d\eta$ is closed, and its class in de Rham cohomology, which is independent of the choices of $\eta$ and $\omega$, is the Godbillon-Vey class of the foliation.  This algorithm is a key tool in the constructions of Roussarie \cite{gv} and Thurston \cite{thurston1} of foliations with nontrivial Godbillon-Vey invariant.

There is recent interest in using a similar Godbillon-Vey algorithm for singular foliations, for example \cite{machon}.
However, additional care is needed in the singular setting.
For example, while it might seem natural to allow all \( 1 \)-forms that satisfy the formal identity $\omega\wedge \d\omega = 0$, once \( \omega \) is allowed to vanish, this condition no longer suffices to guarantee that \( \omega \) is determined by a Haefliger structure.
Indeed, the question of when singular \( q \)-forms have an integrating factor is rather subtle (see Malgrange \cite{malgrange2}), even for codimension \( q = 1 \) (Malgrange \cite{malgrange1}).

For Haefliger foliations the situation is better, and the standard Godbillon-Vey construction can be extended to singular Haefliger structures.
We are not aware of a reference, but this can be considered a classical fact, because it can be shown using the Haefliger structure and a partition of unity.
It is also an easy consequence of our theory, as in the following Proposition.
Recall that a Haefliger structure is \emph{transversely orientable} if its normal bundle is orientable.

\begin{prop}[Singular Godbillon-Vey algorithm]\label{gvalg}
  Let $(M,\FF)$ be a transversely orientable, Haefliger-singular foliation of codimension $q$.
  There is a $q$-form $\omega$ and a \( 1 \)-form \( \eta \) on $M$ for which:
  \begin{enumerate}
    \item $\omega|_{\tilde{M}}$ is nowhere vanishing and defines the regular subfoliation,
    \item $\d\omega = \eta\wedge\omega$, and
    \item $(-1)^{q+1}\,\eta\wedge(\d\eta)^{q}$ is a closed form representing the Godbillon-Vey class of the foliation.
  \end{enumerate}
\end{prop}

\begin{proof}
  Let $h$ denote the Haefliger structure defining the foliation, and fix an adapted geometry $(\epsilon,\nabla)$ for $(M,\FF)$.
  Fix a basis on \( \RB^{q} \), decompose the tautological \( \RB^{q} \)- and \( \glf(\RB^{q}) \)-forms on \( \Fr_{2}(h) \) into components,
  \[ \omega^{1}_{h} = (\omega^{i}), \qquad \omega^{2}_{h} = (\omega^{i}_{j}) \qquad\quad \mbox{ for } i, j = 1, \ldots q , \]
  and let \( \hat{\omega} := \omega^1\wedge\cdots\wedge\omega^q \) and \( \hat{\eta} := -\Tr(\omega^2_h) \).
  From the first structure Equation \eqref{structureequations},
  \begin{align*}
    \d\hat{\omega} &= \sum_{k=1}^{q}(-1)^{k+1}\omega^1\wedge\cdots\wedge \omega^{k-1}\wedge\d\omega^{k}\wedge\omega^{k+1}\wedge\cdots\wedge\omega^{q}\\
    & = \sum_{k=1}^{q} (-1)^{k} \omega^1 \wedge \cdots \wedge \omega^{k-1} \wedge \omega^{k}_{j} \wedge \omega^{j} \wedge \omega^{k+1} \wedge \cdots \wedge \omega^{q} \\
    & = \sum_{k=1}^{q} - \omega^{k}_{k} \wedge \hat{\omega} = \hat{\eta} \wedge \hat{\omega}.
  \end{align*}
  Being $\SO(q)$-basic, the forms $\hat{\omega}$ and \( \hat{\eta} \) descend to $\Fr_{2}(h)/\SO(q)$, and so we may let \( \omega, \eta \) be the respective smooth pullbacks by the adapted section \( \sigma_{\epsilon}\sigma_{\nabla} \).
  It is clear that these satisfy conditions (1) and (2), and condition (3) follows from Theorem \ref{cwsing}, because the Godbillon-Vey invariant is represented by
  \[
    \lambda_{\epsilon,\nabla}(h_{1}c_{1}^{q}) = \Tr(\omega^2_h)\wedge(\d\Tr(\omega^2_h))^{q} = (-1)^{q+1}\eta\wedge(\d\eta)^{q} .
  \]
\end{proof}

\section{Discussion}

Our work opens up a number of directions for future research, which we briefly discuss now.

\noindent\textbf{Noncommutative geometry:} The study of singular foliations has gained substantial traction in recent years, following the seminal construction of the holonomy groupoid of a Stefan-Sussmann singular foliation by I. Androulidakis and G. Skandalis \cite{iakovos2} (as we show in Appendix \ref{app: Stefan-Sussman}, all foliations we consider fall into this class).  As is well-known, the Gel'fand-Fuks classes of a regular foliation play an important role in noncommutative geometry, where they define cyclic cocycles that pair with $K$-theory elements of the convolution algebra associated to the holonomy groupoid to yield numerical invariants \cite{cyctrans}.  We anticipate that our theory will facilitate the construction of analogous cocycles to pair with groupoid algebras of Haefliger-singular foliations, allowing deeper insight into the structure of the noncommutative leaf spaces of such objects.

\noindent\textbf{Construction of new examples:} Chern-Weil theory has long been used in the demonstration of the nontriviality of foliation characteristic classes by constructing appropriate regular foliations \cite{thurston1}.
The combination of Chern-Weil theory with the additional flexibility afforded by singular foliations allows for new constructions.
In a follow-up paper to this we plan to describe new examples of codimension \( 2 \) foliations with continuously varying Godbillon-Vey number.
That construction is made possible by allowing singularities, while the theory in this paper allows us to maintain control over the Godbillon-Vey number.

\appendix

\section{Haefliger structures and Stefan-Sussmann singular foliations}\label{app: Stefan-Sussman}

\newcommand{\vfSheaf}{\mathcal{D}}

Here we fill a small gap in the literature, showing that all smooth Haefliger structures on manifolds induce Stefan-Sussmann foliations.
First we recall the definition of a Stefan-Sussmann foliation, referring to \cite{iakovos2} for more details.

Consider a distribution \( \Delta \) on a smooth manifold \( M \), a fibrewise linear subset \( \Delta \subseteq TM \), where the choice of linear subspace \( \Delta_{x} \subseteq T_{x} M \) may be of varying dimension for each \( x \in M \).
Say that \( \Delta \) is \emph{smooth} if it is spanned by a subsheaf \( \vfSheaf \) of the sheaf \( \XF(M) \) of vector fields on \( M \), meaning that each vector in each \( \Delta_{x} \) extends locally to a vector field in \( \Delta \).
An \emph{integral submanifold} of a smooth distribution is an imbedded submanifold \( \iota \colon \Sigma \to M \) (not necessarily homeomorphic onto its image) that through each image point has tangent plane equal to the distribution, \( \d \iota (T_{x} \Sigma) = \Delta_{\iota(x)} \).
A smooth distribution is \emph{integrable}, or \emph{Stefan-Sussmann}, if it has an integral submanifold through each point;
in this case, the maximal integral submanifolds are unique through each point, and determine a partition of \( M \) into imbedded submanifolds.
Stefan \cite{stefan} and Sussmann \cite{sussmann} independently and near simultaneously discovered the conditions under which a smooth distribution is integrable, the statement as follows.
\begin{thm}[\cite{stefan}, \cite{sussmann}]
	On a smooth manifold \( M \), a smooth distribution \( \Delta \) (spanned by \( \vfSheaf \)) is integrable if and only if \( \Delta \) is invariant under the local flows of (local) vector fields in \( \vfSheaf \).
\end{thm}

Now, let $h \colon \check{\UU}\rightarrow\Gamma_{q}$ be a Haefliger cocycle over an open cover $\UU = \{U_{\alpha}\}_{\alpha\in\AF}$ of a manifold $M$ (Definition \ref{Gammacocycle}), with corresponding Haefliger charts $f_{\alpha}:=s\circ h_{\alpha\alpha} \colon U_{\alpha}\rightarrow\RB^{q}$.
There is a naturally associated differential ideal \( \I \) on \( M \), generated locally in each Haefliger chart by \( \omega_{\alpha}^{i} := f_{\alpha}^{*} \d x^{i} \) (where the \( x^{i} \) are standard coordinates on \( \RB^{q} \)).
The ideal \( \I \) is \emph{formally integrable}, generated algebraically by the 1-forms \( \omega_{\alpha}^{i} \); equivalently, for each \( i = 1, \ldots, q \),
\[ \d \omega_{\alpha}^{i} \equiv 0 \mod{\omega_{\alpha}^{j}, \quad j = 1,\ldots q} . \]
(In fact, \( \d \omega_{\alpha}^{i} = 0 \) identically, but it is the weaker displayed condition that is coordinate independent, and crucial.)
To this ideal is associated a sheaf of vector fields; for \( U \) open in \( M \),
\[ \vfSheaf(U) = \left\{ X \in \XF(U) \colon \omega(X) = 0 \mbox{ all } \omega \in \I(U) \right\} . \]

The sheaf \( \vfSheaf \) spans a smooth distribution \( \Delta \subset TM \), and per the Stefan-Sussmann Theorem, \( \Delta \) is seen to be intregrable once it is shown to be invariant under flows of elements of \( \vfSheaf \).
\begin{prop}
  The distribution \( \Delta \) associated to a Haefliger cocycle on a manifold $M$ is integrable.
\end{prop}
\begin{proof}
  It suffices to show for each point \( x \in M \) and vector field \( X \in \vfSheaf \) defined near \( x \), that \( \d \Phi_{t} (\Delta_{x}) = \Delta_{\Phi_{t}(x)} \), with \( \Phi_{t} = \Phi_{t}^{X} \) the local flow of \( X \).
  But this suffices to be shown for small enough \( t \) that a neighbourhood of the curve \( t \mapsto \Phi_{t}(x) \) is contained in the domain of a single Haefliger chart \( f_{\alpha} \colon U_{\alpha} \to \RB^{q} \).

  Now, the subspace \( \Delta_{y} \) depends only on the stalk of \( \vfSheaf \) at \( y \), which in turn depends only on the stalk of \( \I \) at \( y \).
  The latter is invariant under local flows along \( \vfSheaf \), so the former is too.
  Explicitly, that \( (\Phi_{t}^{*} \I)_{x} = \I_{x} \) follows from the observation that \( f_{\alpha} \composed \Phi_{t} = f_{\alpha} \) and thus
  \[ (\Phi_{t}^{*} \omega_{\alpha}^{i})_{x} = \Phi_{t}^{*} f_{\alpha}^{*} \d x^{i} =  (f_{\alpha} \Phi_{t})^{*} \d x^{i} = f_{\alpha}^{*} \d x^{i} = (\omega_{\alpha}^{i})_{x} .
  \]
\end{proof}

\section{The diffeological generalisation of Haefliger's classifying theorem}\label{app: smooth Haefliger classifying}

In \cite{ha6}, Haefliger uses Milnor's infinite join construction to construct the classifying space \( B\Gamma \) for any topological groupoid $\Gamma$, so that concordance classes of continuous, numerable $\Gamma$-structures on a topological space $X$ are in bijective correspondence with homotopy classes of continuous maps $X\rightarrow B\Gamma$.
This appendix provides both a detailed review and a diffeological rephrasing of definitions and constructions concerning classifying spaces, before proving a generalisation of Haefliger's classifying theorem for diffeological groupoids \( \Gamma \).
It is via diffeological methods that we are able to view our classifying theorem as an honest generalisation of that of Haefliger, and give the de Rham-theoretic universal characteristic map from Gel'fand-Fuks cohomology to the cohomology of the smooth Haefliger classifying space.
Ultimately, this de Rham characteristic map is necessary to facilitate a clean identification of Chern-Weil classes with those coming from the classifying space.
We find that the extension from the topological to the diffeological category requires enough non-trivial  modifications that it is worth providing the details.

The proof of Haefliger's Theorem given here factors through Mostow's smooth structure for classifying spaces \cite{mostow}, which we identify with a diffeology on \( B\Gamma \) that we call the \emph{Mostow diffeology}.
Haefliger's construction can easily be identified with the topological version of Definition \ref{bgamma} \cite[p. 278]{mostow}, provided one equips $B\Gamma$ with the so-called \emph{strong topology} (defined below).
On the other hand, the Mostow diffeology induces its D-topology on \( B\Gamma \), which has fewer opens than the strong topology.
Since any smooth map of a diffeological \( X \) into \( B\Gamma \) is automatically continuous with respect to the D-topology, the smooth map is also continuous for the strong topology.
In other words, smooth maps to \( B\Gamma \) will by definition be continuous in Haefliger's sense.

An alternative diffeological approach to classifying spaces, which is insufficient for our purposes, is presented in \cite{magnot3}.

\subsection{Diffeology}
We begin by briefly recalling the framework of diffeology, referring to the book \cite{diffeology} for details. A \emph{diffeological space} is a set \( X \) with a \emph{smooth structure}, determined by a collection of \emph{plots} \( V \to X \), as \( V \) ranges over the open subsets of finite dimensional Euclidean spaces.
This collection of plots, called a \emph{diffeology}, must satisfy three axioms: constant maps are plots, all maps that are locally given by plots are plots, and closure under precomposition by smooth functions between open subsets of Euclidean space.
(The map from the empty set is vacuously locally a plot, so a plot).
A diffeology on \( X \) determines a topology on \( X \), \emph{the D-topology}, which is the finest topology for which all plots are continuous---a subset \( U \subset X \) is D-open if and only if \( P^{-1}(U) \) is open for all plots \( P \) of \( X \).
Unless otherwise stated, diffeological spaces will always be assumed to carry the D-topology.
Conversely, any topological space carries a natural diffeology, called the \emph{continuous diffeology}, whose plots are the continuous maps.

A map $f \colon X_{1}\rightarrow X_{2}$ of diffeological spaces is \emph{smooth} if \( f \composed P \) is a plot of \( X_2 \) for each plot \( P \) of $X_{1}$.
By definition, smooth maps between diffeological spaces are continuous with respect to the underlying \( D \)-topologies.
Given a \( D \)-open subset \( U \subset X_1 \), a map \( f \colon U \to X_2 \) is \emph{locally smooth} if, for each plot \( P \) of \( X_1 \), the composition \( P^{-1}(U) \to U \to X_2 \) is a plot of \( X_2 \).
Remark that local smoothness of \( f \) depends on all plots of \( X_{1} \), not just the ones with image in \( U \).

Subsets and quotients of a diffeological space \( X \) carry natural diffeologies (the \emph{subspace} and \emph{quotient} diffeologies).
Furthermore, the category of diffeological spaces and smooth maps is both complete and cocomplete with respect to limits, and contains the category of manifolds and smooth maps as a full, faithful subcategory.
Finally, any diffeological space $X$ has an associated de Rham complex $(\Omega^{*}(X),\d)$, which is contravariantly functorial in the manner familiar from manifold theory, and which coincides with the usual de Rham complex when $X$ is a manifold.

\subsection{Semi-simplicial sets}

We now aim towards the simplicial set construction of classifying spaces.
A \emph{semi-simplicial object} in a category $\CC$ is a sequence $X^{\bullet} = \{X^{(n)}\}_{n\in\NB}$ of objects in $\CC$ together with morphisms $\partial_{i} \colon X^{(n+1)}\rightarrow X^{(n)}$ defined for $i=0,\dots,n+1$, called \emph{face maps}, satisfying the relations
\[ \partial_{i}\circ\partial_{j} = \partial_{j-1}\circ\partial_{i},\qquad\text{ for $i<j$}. \]
A \emph{morphism} of semi-simplicial objects $X^{\bullet}$ and $Y^{\bullet}$ in $\CC$ is a natural transformation, a family $\phi^{\bullet} = \{\phi^{(n)} \colon X^{(n)}\rightarrow Y^{(n)}\}$ of morphisms that commute with respective face maps.  Moreover, if $\CC$ is a Cartesian category, and $X^{\bullet}$ and $Y^{\bullet}$ are two semi-simplicial objects in $\CC$, then their product $X^{\bullet}\times Y^{\bullet}$ is also a semi-simplicial object in $\CC$, with face maps $\partial^{X\times Y}_{i}:=\partial_{i}^{X}\times\partial_{i}^{Y}$.  The following examples are key for the classifying space construction.

\begin{ex}
  Any diffeological groupoid $\Gamma$ has a natural associated semi-simplicial space $\Gamma^{\bullet}$, its \emph{nerve}, with $\Gamma^{(n)}$ the subspace of composable \( n \)-tuples in $(\Gamma^{(1)})^{n}$ and the face maps $\partial_{i} \colon \Gamma^{(n+1)}\rightarrow\Gamma^{(n)}$ defined by the formul\ae
  \[
    \partial_{i}(\gamma_{1},\dots,\gamma_{n+1}):=
    \begin{cases}
      (\gamma_{2},\dots,\gamma_{n+1}) & \text{ for } i=0 , \\
      (\gamma_{1},\dots,\gamma_{i}\gamma_{i+1},\dots,\gamma_{n+1}) & \text{ for } 1 \leq i \leq n \\
      (\gamma_{1},\dots,\gamma_{n})  & \text{ for } i=n+1
    \end{cases}
  \]
  for $n\geq1$.
  The zero-tuples \( \Gamma^{(0)} \) are the objects (identity morphisms) of \( \Gamma \), and the maps $\partial_{0},\partial_{1} \colon \Gamma^{(1)}\rightarrow\Gamma^{(0)}$ are the source and range maps respectively.
  Any morphism of groupoids induces an obvious morphism of their nerves.
  In particular, the map \( h \colon \check{\UU}\rightarrow\Gamma \) defining a smooth $\Gamma$-cocycle determines, by a mild abuse of notation, a map $h \colon \check{\UU}^{\bullet}\rightarrow\Gamma_{q}^{\bullet}$ of semi-simplicial diffeological spaces.

  Also useful are the smooth maps \( g_{i} \colon \Gamma^{(n)} \to \Gamma \) for each \( 0 \le i \le n \) that send an element of $\Gamma^{(n)}$ to the composition of its first \( i \) elements,
  \begin{align*}
    g_0(\gamma_{1},\dots,\gamma_{n}) & := s(\gamma_{1}) \\
    g_{i}(\gamma_{1},\dots,\gamma_{n}) & := \gamma_1 \cdots \gamma_{i} \quad \mbox{ for } i > 0 .
  \end{align*}
  Likewise, for a pair of such indices \( i, j \), define the smooth
  \[ g_{ij}(\gamma_{1},\dots,\gamma_{n}) = g_{i}^{-1}(\gamma_{1},\dots,\gamma_{n}) g_{j}(\gamma_{1},\dots,\gamma_{n}) . \]

\end{ex}

\begin{ex}
  Consider the natural numbers $\NB$ (including \( 0 \)), equipped with their zero-dimensional manifold structure.
  Define a semi-simplicial space $\NB^{\bullet}$ by setting $\NB^{(k)}$ to be the set of all strictly increasing $k+1$-tuples $\alpha_{0}<\cdots< \alpha_{k}$ of natural numbers.  The face maps $\partial_{j} \colon \NB^{(k)}\rightarrow\NB^{(k-1)}$ are given by omission: $(\alpha_{0},\dots,\alpha_{k})\mapsto (\alpha_{0},\dots,\hat{\alpha_{j}},\dots,\alpha_{k})$.
\end{ex}

Associated to any semi-simplicial diffeological space $X^{\bullet}$ is a diffeological space, the \emph{fat realisation} \( \|X^{\bullet}\| \).
For each $n\in\NB$, let $\Delta_{n}$ denote the standard $n$-simplex, thought of as a diffeological subspace of Euclidean space.
Let $(t_{0},\dots,t_{n})$ denote the barycentric coordinates on $\Delta_{n}$,
and denote by $d_{i} \colon \Delta_{n}\rightarrow\Delta_{n+1}$ the $i^{th}$ face inclusion.
The fat realisation is the quotient
\[
\|X^{\bullet}\|:=\bigg(\bigsqcup_{n\in\NB}\Delta_{n}\times X^{(n)}\bigg)/\sim
\]
by the relation $(d_{i}(\vec{t}),x)\sim(\vec{t},\partial_{i}(x))$ for any choice of $(\vec{t},x)\in\Delta_{n}\times X^{(n+1)}$, $n\geq0$ and $i=0,\dots,n+1$.

There are two natural families of maps on \( \|\NB^{\bullet}\times\Gamma^{\bullet}\| \).
First, for each $\alpha\in\NB$ there is a barycentric coordinate function $u_{\alpha} \colon \|\NB^{\bullet}\times\Gamma^{\bullet}\|\rightarrow[0,1]$, defined by the following family of functions
\[
\Delta_n \times \bigl( \NB^{(n)}\times\Gamma^{(n)} \bigr) \ni  (t_{0},\dots,t_{n};\alpha_{0},\dots,\alpha_{n};\gamma_{1},\dots,\gamma_{n}) \xmapsto{\qquad}
\begin{cases} t_{j} \quad \text{ if $\alpha = \alpha_{j}$}\\ 0 \quad \text{ otherwise}\end{cases} ,
\]
which are compatible with the the equivalence relation defining the fat realisation, so descend to a well-defined function.
Note that the sets $U_{\alpha}:=u_{\alpha}^{-1}(0,1]$ define a cover of \( \|\NB^{\bullet}\times\Gamma^{\bullet}\| \) (which will be an open cover in the \( D \)-topology of the Mostow diffeology defined below).
Second are the transition maps
$\gamma_{\alpha\beta} \colon U_{\alpha}\cap U_{\beta}\rightarrow\Gamma$, given by
\[ \gamma_{\alpha\beta}([t_{0},\dots,t_{n};\alpha_{0},\dots,\alpha_{n};\gamma_{1},\dots,\gamma_{n}]) := g_{ij}(\gamma_{1},\dots,\gamma_{n}) \qquad \mbox { for } \alpha_i = \alpha \mbox{ and } \alpha_j = \beta .
\]
\subsection{Simplicial realization of classifying spaces}
The fat realization can be used to construct the classifying space \( B \Gamma \) of a diffeological groupoid \( \Gamma \).
As a set, \( B\Gamma = \|\NB^{\bullet}\times\Gamma^{\bullet}\| \), but it will be useful to consider the \emph{Mostow diffeology}, a coarser diffeology than the quotient diffeology of the fat realization.
Being a coarser diffeology means that the set of Mostow plots contains the set of quotient plots, which in turn means that \( B\Gamma \) with the Mostow diffeology admits more smooth maps into.
While Mostow never invokes diffeology explicitly in \cite{mostow}, the Mostow diffeology defined here is implicit in his constructions, via the definition of differential forms \cite[Section 2]{mostow}.

The \emph{Mostow diffeology} is the largest diffeology (\cite[Section 1.24]{diffeology}) for which the maps $u_{\alpha} \colon \|\NB^{\bullet}\times\Gamma^{\bullet}\|\rightarrow[0,1]$ are smooth and the maps $\gamma_{\alpha\beta} \colon U_{\alpha}\cap U_{\beta}\rightarrow\Gamma$ are locally smooth.
With this diffeology, any map from a diffeological space \( \eta \colon X \to B\Gamma \) can be shown to be smooth by checking for all \( \alpha, \beta \in \NB \) that \( u_{\alpha} \circ \eta \) and \( \gamma_{\alpha\beta} \circ \eta_{\alpha\beta} \) are smooth, where \( \eta_{\alpha\beta} \) is the restriction to \( \eta^{-1}(U_{\alpha} \cap U_{\beta}) \).
In particular, it is not difficult from the definitions that every plot of the quotient diffeology on \( B\Gamma \) is a Mostow plot.
Essentially, the quotient plots are those that locally factor through the finite simplex sets \( \Delta_{n} \times \NB^{(n)} \times \Gamma^{(n)} \), and the \( u_{\alpha}, \gamma_{\alpha\beta} \) are (locally) smooth on each of these.

We call the the D-topology of the Mostow diffeology the \emph{Mostow topology} on \( B\Gamma \).
The maps \( u_{\alpha} \) and \( \gamma_{\alpha\beta} \) can also be used to define the \emph{strong topology}, as the smallest topology making $u_{\alpha}$ and $\gamma_{\alpha\beta}$ continuous.
Otherwise put, the preimages of opens in \( [0,1] \) by the \( u_{\alpha} \) and opens in \( \Gamma \) by the \( \gamma_{\alpha\beta} \) comprise a subbasis for the strong topology.
With the latter characterization, the strong topology on $B\Gamma$ is easily seen to be coarser than the Mostow topology, because smooth maps are always continuous with respect to the D-topology.

\begin{defn}\label{bgamma}
  Let $\Gamma$ be a diffeological groupoid. The \textbf{classifying space} $B\Gamma$ of $\Gamma$ is the diffeological space whose underlying set is $\|\NB^{\bullet}\times\Gamma^{\bullet}\|$, equipped with the Mostow diffeology.
  Call the open cover $\US:=\{U_{\alpha}\}_{\alpha\in\NB}$ the \textbf{canonical open cover}, and the smooth $\Gamma$-cocycle $\gamma \colon \check{\US}\rightarrow\Gamma$ the \textbf{canonical cocycle}.
  The Haefliger structure determined by $\gamma$ is the \textbf{canonical $\Gamma$-structure} on \( B\Gamma \).
\end{defn}

It is immediate from the definition that a smooth map \( \eta \colon X \to B\Gamma \) induces a \( \Gamma \)-cycle on \( X \), by pulling back the canonical \( \Gamma \)-structure on \( B\Gamma \).
It also follows from the definitions that classifying spaces are functorial: for any morphism $\phi \colon \Gamma_{1}\rightarrow\Gamma_{2}$ of diffeological groupoids, there is a smooth map $B\phi \colon B\Gamma_{1}\rightarrow B\Gamma_{2}$, which furthermore preserves the canonical cocycles: $u^{2}_{\alpha}\circ B\phi = u_{\alpha}^{1}$ and $\phi\circ\gamma^{1}_{\alpha\beta} = \gamma^{2}_{\alpha\beta}\circ B\phi$, with $u_{\alpha}^{i}$ and $\gamma_{\alpha\beta}^{i}$ the maps defining the canonical cocycle of $B\Gamma_{i}$.
Just as in the topological setting, Definition \ref{bgamma} is a special case of a more general construction, which assigns to any semi-simplicial diffeological $X^{\bullet}$ its \emph{unwound geometric realisation} $\mu(X^{\bullet})$ (cf. \cite{segal, tomdieck, mostow}).

\begin{ex}\label{ex: POU}
  Already the case of \( \Gamma = \{e\} \), the trivial groupoid, is interesting.
  The classifying space \( B \{e\} \) can be identified with the infinite simplex \( \Delta_{\infty} \), comprising the infinite non-negative sequences $t = \{t_{\alpha}\}_{\alpha\in\NB}$ for which only finitely many elements are nonzero and $\sum_{\alpha}t_{\alpha} = 1$.
  Under this identification, the \( u_{\alpha} \) are precisely the barycentric coordinates of \( \Delta_{\infty} \), which determine a canonical pointwise finite smooth partition of unity on \( \Delta_{\infty} \).
  From this identification, a choice of smooth map \( X \to B \{e\} \) is the same as a choice of pointwise finite smooth partition of unity on \( X \).

  The partition of unity can be improved.
  Mostow \cite[p. 273]{mostow} constructs from the barycentric coordinates a locally finite smooth partition of unity \( \{v_{\alpha}\}_{\alpha\in\NB} \) that is subordinate to the canonical cover of \( \Delta_{\infty} \).
  Furthermore, the induced map \( v = (v_{\alpha}) \colon \Delta_{\infty} \to \Delta_{\infty} \) is smoothly homotopic to the identity by the usual straight line homotopy.
  Fix once and for all the choice of \( (v_{\alpha}) \) on \( \Delta_{\infty} \), and denote also by \( (v_{\alpha}) \) the pullbacks to each \( B\Gamma \) by the map \( B\Gamma \to B \{ e \} \).

  These considerations show the following Lemma.
  \begin{lemma}\label{thm: homotopy to locally finite}
    The canonical $\Gamma$-structure $\gamma$ on $B\Gamma$ is smoothly numerable.
    Furthermore, every smooth \( \eta' \colon X \to B\Gamma \) is smoothly homotopic to \( \eta \) for which the pullback open cover \( \{ \eta^{-1}(U_{\alpha}) \} \) is locally finite, and the pullback partition of unity \( \{ \eta^{*} v_{\alpha} \} \) is subordinate to the pullback cover.
  \end{lemma}
\end{ex}

\subsection{The classifying map}
To justify the claim that \( B\Gamma \) is a smooth classifying space, it must be shown that smooth maps into it classify smooth \( \Gamma \)-structures (Definition \ref{Gammacocycle}).
This is the content of the following Theorem.

\begin{thm}\label{Appclassify}
  Let $\Gamma$ be a diffeological groupoid.
  \begin{enumerate}
  \item
    For any smoothly numerable $\Gamma$-structure $h$ on a diffeological space $X$, there is a smooth map $\eta \colon X\rightarrow B\Gamma_{q}$ such that $h = \eta^{*}\gamma$.

  \item
    If $\eta_{0},\eta_{1} \colon X\rightarrow B\Gamma$ are smooth maps, then the $\Gamma$-structures $\eta_{0}^{*}\gamma$ and $\eta_{1}^{*}\gamma$ are smoothly, numerably concordant if and only if $\eta_{0}$ and $\eta_{1}$ are smoothly homotopic.
  \end{enumerate}
\end{thm}
\begin{proof}
  (1)
  Let $h \colon \check{\VV}\rightarrow\Gamma$ be a $\Gamma$-cocycle for $X$, defined over a countable open cover $\VV = \{V_{\alpha}\}_{\alpha\in\NB}$ with smooth, subordinate, locally finite partition of unity $\{\lambda_{\alpha}\}_{\alpha\in\NB}$.
  Call the transition maps of \( h \) by \( h_{\alpha\beta} \colon V_{\alpha} \cap V_{\beta} \to \Gamma \).
  For each point \( x \) of \( X \), let $A_{x} := (\alpha_{0},\dots,\alpha_{n})$ be the finite, ordered list of indices supporting \( x \), those indices for which \( x \in \supp\lambda_{\alpha} \).
  Choose about each \( x \) a small D-open set \( W_{x} \) such that
  \begin{equation}\label{eq: good small open}
    \begin{aligned}
      W_{x} \cap \supp(\lambda_{\alpha}) = \varnothing \quad \mbox{ if } \alpha \not \in A_{x} , \\
      W_{x} \subset V_{\alpha} \quad \mbox{ if } \alpha \in A_{x} ,
    \end{aligned}
  \end{equation}
  which choice can be made because the partition of unity is locally finite;
  choose any open set about \( x \) that meets only finitely many supports, subtract the supports that don't meet \( x \), and then intersect with the finitely many opens \( V_{\alpha} \) for \( \alpha \in A_{x} \).
  Note that any \( y \in W_{x} \) is supported on the same indices as \( x \), so that \( A_{y} \subseteq A_{x} \).
  For each \( x \), with \( A_{x} = (\alpha_{0}, \ldots, \alpha_{n}) \), define the map \( \eta_{x} \colon W_{x} \to B\Gamma \)
  by formula
  \begin{equation}\label{eta}
    \eta_{x}(y) := [(\lambda_{\alpha_{0}}(y),\dots,\lambda_{\alpha_{n}}(y);\alpha_{0},\dots,\alpha_{n};h_{\alpha_{0}\alpha_{1}}(y),\dots,h_{\alpha_{n-1}\alpha_{n}}(y))] , \qquad y \in W_{x} ,
  \end{equation}
  which is well-defined by the choice of \( W_{x} \).

  The maps \( \eta_{x} \) agree on overlaps of domains, so can be glued to a single map of sets \( \eta \colon X \to B\Gamma \).
  To check this, we require that \( \eta_{x_{1}}(y) = \eta_{x_{2}}(y) \) for any point \( y \) in the intersection of two domains, \( y \in W_{x_{1}} \cap W_{x_{2}} \).
  It suffices, by symmetry in \( x_{1}, x_{2} \), to check that \( \eta_{x}(y) = \eta_{y}(y) \) for any \( y \in W_{x} \).
  As noted above, this means that \( A_{y} \subseteq A_{x} \), from which it is not difficult to see that \( \eta_{x}(y) \) is equivalent to \( \eta_{y}(y) \) in \( B\Gamma \) by the sequence of face maps that eliminate the zeros corresponding to the indices \( \alpha \in A_{x} \setminus A_{y} \).

  To show that \( \eta \) is smooth, it suffices to show that each restriction \( \eta_{x} \) is smooth, and to lighten notation, we may suppose that \( \eta = \eta_{x} \) and \( X = W_{x} \).
  As noted at the definition of the Mostow diffeology, \( \eta \) is smooth if \( u_{\alpha} \circ \eta \) and \( \gamma_{\alpha\beta} \circ \eta_{\alpha\beta} \) are smooth for all \( \alpha, \beta \in \NB \), with \( \eta_{\alpha\beta} \) the restriction of \( \eta \) to the preimage of \( U_{\alpha} \cap U_{\beta} \).
  It is immediate from the definition that
  \[ u_{\alpha} \circ \eta = \lambda_{\alpha} , \]
  so each \( u_{\alpha} \circ \eta \) is smooth.
  It also follows from this equality that the sets 
  \( \eta^{-1}(U_{\alpha}\cap U_{\beta}) = \lambda_{\alpha}^{-1}((0,1]) \cap \lambda_{\beta}^{-1}((0,1]) \)
  are open in \( X \), and non-empty only if \( \alpha, \beta \in A_{x} \).
  The map from the empty set is smooth, so suppose that \( \alpha, \beta \in A_{x} \).
  There are indices \( i, j \) such that \( \alpha_{i} = \alpha \) and \( \alpha_{j} = \beta \), and supposing \( i < j \),
  \[
    \gamma_{\alpha\beta} \circ \eta_{\alpha\beta} = \big(h_{\alpha_{0}\alpha_{1}}\cdots h_{\alpha_{i-1}\alpha_{i}}\big)^{-1}\big(h_{\alpha_{0},\alpha_{1}}\cdots h_{\alpha_{j-1},\alpha_{j}}\big) = h_{\alpha_{i}\alpha_{j}} = h_{\alpha\beta}
  \]
  is smooth because the transition functions \( h_{\alpha\beta} \) are smooth.

  Finally, the last two display equations hold true for the original classifying map \( \eta \) defined on all of \( X \).
  From this, it is easy that the \( \Gamma \)-structures \( h \) and \( h' = \eta^{*}\gamma \) are equivalent, because the pullback cover is given by opens \( V'_{\alpha} = \eta^{-1}(U_{\alpha}) = \lambda_{\alpha}^{-1}((0,1]) \subset V_{\alpha} \), and \( \eta^{*}_{\alpha\beta}\gamma_{\alpha\beta} = h_{\alpha\beta} \) upon restricting to \( V'_{\alpha} \cap V'_{\beta} \).

  (2)
  Suppose that $\eta_{0},\eta_{1} \colon X\rightarrow B\Gamma$ are smooth maps.
  If $\eta_{0}$ and $\eta_{1}$ are smoothly homotopic, say via a smooth \( \eta \colon X \times [0, 1] \to B\Gamma \), then up to a smooth homotopy, the induced \( \eta^{*} \gamma \) determines a smooth numerable concordance between \( \eta_{0}^{*} \gamma \) and \( \eta_{1}^{*} \gamma \) (Lemma \ref{thm: homotopy to locally finite}).
  Suppose conversely that $\eta_{0}^{*}\gamma$ and $\eta_{1}^{*}\gamma$ are smoothly, numerably concordant through a $\Gamma$-structure on $X\times[0,1]$.
  By item (1), this concordance is of the form $\eta^{*}\gamma$ for some smooth $\eta \colon X\times[0,1]\rightarrow B\Gamma$, so to prove (2) it suffices to show that \( \eta_{0} \) is homotopic to \( \eta|_{X \times \{0\}} \).
  In other words, we must show that $\eta_{0}^{*}\gamma = \eta_{1}^{*}\gamma$ implies that \( \eta_{0}, \eta_{1} \) are smoothly homotopic.
  We may furthermore assume, by Lemma \ref{thm: homotopy to locally finite}, that the \( \eta_{i} \) induce locally finite open covers.

  Our argument is inspired by those found in \cite[p.57-58]{husemoller} and in \cite[Proposition 3.16]{magnot3}.
  Let $B\Gamma^{0}$ and $B\Gamma^{1}$ denote the subsets of $B\Gamma$ consisting of tuples $[\vec{t},\vec{\alpha};\vec{\gamma}]$ for which $\vec{\alpha}$ consists entirely of even or odd numbers respectively.  Replacing the linear functions $\alpha_{n} \colon I_{n}:=[1-2^{-n}, 1-2^{-n-1}]\rightarrow[0,1]$ of \cite[p. 57]{husemoller} with smooth functions $b_{n} \colon I_{n}\rightarrow[0,1]$ which are everywhere nondecreasing, and constant on a small neighbourhood of each endpoint, the  arguments of \cite[p.57-58]{husemoller} can be used to show that the maps $h^{1} \colon B\Gamma\rightarrow B\Gamma$ and $h^{0} \colon B\Gamma\rightarrow B\Gamma$ defined by
  \begin{align*}
    h^{0}([t_{0},\dots,t_{k};\alpha_{0},\dots,\alpha_{k};\gamma_{1},\dots,\gamma_{k}]) & := [t_{0},\dots,t_{k};2\alpha_{0},\dots,2\alpha_{k};\gamma_{1},\dots,\gamma_{k}] , \\
    h^{1}([t_{0},\dots,t_{k};\alpha_{0},\dots,\alpha_{k};\gamma_{1},\dots,\gamma_{k}]) & := [t_{0},\dots,t_{k};2\alpha_{0}+1,\dots,2\alpha_{k}+1;\gamma_{1},\dots,\gamma_{k}], 
  \end{align*}
  are both smoothly homotopic to the identity.
  Hence we may replace $\eta_{i}$ with $h^{i}\circ\eta_{i}$, with values in $B\Gamma^{i}$, so that $\eta_{0}^{*}\gamma$ is defined over the cover $\{V_{2\alpha}:=\eta_{0}^{-1}(U_{2\alpha})\}_{\alpha\in\NB}$ while $\eta_{1}^{*}\gamma$ is defined over the cover $\{V_{2\alpha+1}:=\eta_{1}^{*}(U_{2\alpha+1})\}_{\alpha\in\NB}$.
  By assumption, $\eta_{0}^{*}\gamma$ and $\eta_{1}^{*}\gamma$ are equivalent, so there exists a \( \Gamma \)-cocycle on the union cover $\VV = \{V_{\alpha}\}_{\alpha\in\NB}$ with transition maps $h$ compatible with the transition maps of $\eta_{0}^{*}\gamma$ and $\eta_{1}^{*}\gamma$.

  Let \( \{ \lambda^{i}_{\alpha} := \eta_{i}^{*} v_{\alpha} \} \) be the respective pullback partitions of unity, which are locally finite and subordinate to the respective open covers.
  For each $x\in X$, fix an open \( W_{x} \) satisfying \eqref{eq: good small open} with respect to both partitions of unity, where the index sets \( A^{0}_{x} = (\alpha^{0}_{0},\dots,\alpha^{0}_{n_{0}}) \) and \( A^{1}_{x} = (\alpha^{1}_{0},\dots,\alpha^{1}_{n_{1}}) \) are as there.
  Restricted to \( W_{x} \), the maps \( \eta_{i} \) are
  \[
    \eta_{i}(y) = [\lambda^{i}_{0}(y),\dots,\lambda^{i}_{n_{i}}(y);\alpha^{i}_{0},\dots,\alpha^{i}_{n_{i}};h_{\alpha^{i}_{0}\alpha^{i}_{1}}(y),\dots,h_{\alpha^{i}_{n_{i}-1}\alpha^{i}_{n_{i}}}(y)], \qquad y \in W_{x} .
  \]
  The index sets \( A^{i}_{x} \) are disjoint, and their union \( A_{x} \) has length $n = n_{0}+n_{1}$.
  For $i=0,\dots,n$, define $t_{i} \colon W_{x}\times[0,1]\rightarrow[0,1]$ by
  \[
    t_{i}(y,s):=\begin{cases} (1-s)\,\lambda^{0}_{j}(y)&\text{ if $\alpha_{i} = \alpha^{0}_{j}$} \\ s\,\lambda^{1}_{j}(y)&\text{ if $\alpha_{i} = \alpha^{1}_{j}$}\end{cases} .
  \]
  Then the formula
  \[
    \eta(y,s) = [t_{0}(y,s),\dots,t_{n}(y,s);\alpha_{0},\dots,\alpha_{n};h_{\alpha_{0}\alpha_{1}}(y),\dots,h_{\alpha_{n-1}\alpha_{n}}(y)],\qquad y\in W_{x}.
  \]
  defines a homotopy $\eta \colon W_{x} \times[0,1]\rightarrow B\Gamma$.
  The homotopy is manifestly smooth on each \( W_{x} \), and an argument analogous to that in the proof of (1) demonstrates that they glue together to a smooth homotopy between $\eta_{0}$ and $\eta_{1}$ over all of \( X \).
\end{proof}

Finally, we detail our claim that our diffeological classifying theorem (Theorem \ref{Appclassify}) generalises Haefliger's topological classifying theorem, recalled here.

\begin{thm}[Haefliger's classifying theorem]\cite[Theorem 7]{ha6}\label{haefyclass}
  Let $\Gamma$ be a topological groupoid, and regard $B\Gamma$ with the strong topology.
  \begin{enumerate}
  \item
    For any numerable $\Gamma$-structure $h$ on a topological space $X$, there is a continuous map $\eta \colon X\rightarrow B\Gamma$ such that $h = \eta^{*}\gamma$.
  \item
    If $\eta_{0},\eta_{1} \colon X\rightarrow B\Gamma$ are continuous maps, then $\eta_{0}^{*}\gamma$ and $\eta_{1}^{*}\gamma$ are numerably concordant if and only if $\eta_{0}$ and $\eta_{1}$ are homotopic.
	\end{enumerate}
\end{thm}

To see that Theorem \ref{haefyclass} follows from Theorem \ref{Appclassify}, equip the topological groupoid $\Gamma$ with the continuous diffeology.  Then the canonical partition of unity on the diffeological space $B\Gamma$ is smooth, hence continuous with respect to the D-topology, hence continuous with respect to the strong topology, which is contained in the D-topology.  That the second and third items of Haefliger's classifying theorem follow from the corresponding items of Theorem \ref{Appclassify} can be seen immediately by the application of the following lemma, whose proof is an elementary consequence of the definitions.

\begin{lemma}\label{difftop}
	Let $\Gamma$ be a topological groupoid, equipped with the continuous diffeology.  Then the corresponding Mostow diffeology on $B\Gamma$ coincides with the continuous diffeology on $B\Gamma$ induced by the strong topology.  Consequently, if $X$ is a topological space with the continuous diffeology, a map $\eta \colon X\rightarrow B\Gamma$ is smooth if and only if it is strongly continuous.
\end{lemma}

Suppose now that $h$ is a numerable $\Gamma$-structure on a topological space $X$.  Equipping $X$ with the continuous diffeology, $h$ is smoothly numerable, hence by Theorem \ref{Appclassify} is associated to a smooth (hence strongly continuous) map $\eta \colon X\rightarrow B\Gamma$ such that $h=\eta^{*}\gamma$.  Conversely, by Lemma \ref{difftop}, any strongly continuous map $\eta \colon X\rightarrow B\Gamma$ is automatically smooth when $X$ is equipped with the continuous diffeology and $B\Gamma$ with the Mostow diffeology, and Theorem \ref{Appclassify} then applies to yield a corresponding continuous $\Gamma$-structure $\eta^{*}\gamma$ on $X$.

Finally suppose that $\eta_{0},\eta_{1} \colon X\rightarrow B\Gamma$ are strongly continuous maps.  By Lemma \ref{difftop} they are then smooth for the continuous diffeology on $X$ and the Mostow diffeology on $B\Gamma$, and Theorem \ref{Appclassify} then applies to show that $\eta_{0}^{*}\gamma$ and $\eta_{1}^{*}\gamma$ are numerably concordant if and only if $\eta_{0}$ and $\eta_{1}$ are homotopic.

\section{\texorpdfstring{Sections of Haefliger bundles modulo $\O(q)$}{Sections of Haefliger bundles modulo O(q)}}\label{app: sections of Haefliger bundles}

Here we describe an explicit construction, inspired by Dupont \cite{dupont2}, of a smooth section of the bundle $\Fr_{\infty}(\gamma)/\O(q)\rightarrow B\Gamma_{q}$, from which the existence of analogous sections for all smoothly numerable Haefliger structures follows. Note that the existence of continuous sections follows from the fact that $B\Gamma_{q}$ (with the strong topology) has the homotopy type of a CW-complex together with contractibility of the fibre, but the existence of smooth sections in the diffeological setting requires additional argument.

To this end, we show how to construct a section \( \sigma \) of \( \Fr_{\infty}(\gamma)/\O(q) \) from any section \( \sigma_0 \) of \( \Fr_{\infty}(\RB^q) \to \RB^q \).
In fact, we will identify \( \Fr_{\infty}(\gamma) \) with the classifying space \( B\Gamma_{q,\infty} \) of a smooth groupoid \( \Gamma_{q,\infty} \), and then \( \sigma \) will given functorially as \( B\sigma_0 \).
The construction will furthermore show that any two smooth sections of a Haefliger bundle mod $\O(q)$ are smoothly homotopic, so that their induced characteristic maps (Definition \ref{universalmap}) coincide.

\begin{prop}
	For $k\in\NB\cup\{\infty\}$, define $\Gamma_{q,k}$ to be the action groupoid $\Gamma_{q}\ltimes\Fr_{k}(\RB^{q})$.
	Then $B\Gamma_{q,k}\rightarrow B\Gamma_{q}$ is a principal $G_{q}^{k}$-bundle over $B\Gamma_{q}$ that is canonically isomorphic to $\Fr_{k}(\gamma)\rightarrow B\Gamma_{q}$.
\end{prop}

\begin{proof}
	The groupoid $\Gamma_{q,k} = \Gamma_{q}\ltimes\Fr_{k}(\RB^{q}) = \Gamma_{q}\times_{s,\pi_{k}}\Fr_{k}(\RB^{q})$ is equipped with the subspace diffeology of the product $\Gamma_{q}\times\Fr_{k}(\RB^{q})$,
	meaning that a parameterisation $P \colon U\rightarrow\Gamma_{q}\ltimes\Fr_{k}(\RB^{q})$ is a plot if and only if each of its component maps $U\rightarrow\Gamma_{q}$ and $U\rightarrow\Fr_{k}(\RB^{q})$ are plots.

	A smooth action of a diffeological group $G$ on a diffeological space $X$ is by definition principal if and only if the action map $a \colon X\times G\ni(x,g)\mapsto(x,x\cdot g)\in X\times X$ is a diffeological induction \cite[Section 8.11]{diffeology}, meaning that $a$ is injective (the action is free) and each parameterisation $P \colon U\rightarrow X\times G$ is a plot if and only if $a\composed P$ is.
	Since $\Fr_{k}(\RB^{q})\rightarrow\RB^{q}$ is a principal $G_{q}^{k}$-bundle, the action map $a \colon \Fr_{k}(\RB^{q})\times G_{q}^{k}\rightarrow\Fr_{k}(\RB^{q})\times\Fr_{k}(\RB^{q})$ is an induction.
	It induces an action map $a_{\Gamma} \colon \Gamma_{q,k}\times G^{k}_{q}\rightarrow \Gamma_{q,k}\times\Gamma_{q,k}$ by the rule
	\[
	a_{\Gamma}(\gamma, \phi,g) = \big((\gamma,\phi), (\gamma,a(\phi,g))\big),
	\]
	which is inductive because $a$ is and by definition of the diffeology on $\Gamma_{q,k}$.
	The map $a_{\Gamma}$ induces in turn an action map $Ba_{\Gamma} \colon B\Gamma_{q,k}\times G^{k}_{q}\rightarrow B\Gamma_{q,k}\times B\Gamma_{q,k}$, given by the formula
	\[
	Ba_{\Gamma}\Bigl(\bigl[\vec{t};\vec{\alpha};\overrightarrow{(\gamma,\phi)}\bigr], g\Bigr)
	:=\Bigl(\bigl[\vec{t};\vec{\alpha};\overrightarrow{(\gamma,\phi)}\bigr], \bigl[\vec{t};\vec{\alpha};\overrightarrow{a_{\Gamma}(\gamma,\phi,g)}\bigr]\Bigr) .
	\]
	Here the composable tuple
	\[
	\overrightarrow{(\gamma,\phi)} = \big((\gamma_{1},\phi_{1}),\dots,(\gamma_{n},\phi_{n})\big) \in \Gamma_{q,k}^{(n)}
	\]
	is mapped to the composable tuple
	\[ \overrightarrow{a_{\Gamma}(\gamma,\phi,g)}
	= \big((\gamma_{1},a(\phi_{1},g)),\dots,(\gamma_{n},a(\phi_{n},g))\big) .
	\]

	The map $a_{\Gamma}$ preserves each open $U_{\alpha}$ of the canonical cover of $B\Gamma_{q,k}$, and is such that for any $\alpha,\beta\in\NB$, the diagram
	\begin{center}
		\begin{tikzcd}
			(U_{\alpha}\cap U_{\beta})\times G^{k}_{q} \ar[r, "Ba_{\Gamma}"] \ar[d, "\gamma_{\alpha\beta}^{k}\times\id"] & (U_{\alpha}\cap U_{\beta})\times (U_{\alpha}\cap U_{\beta}) \ar[d, "\gamma_{\alpha\beta}^{k}\times\gamma_{\alpha\beta}^{k}"] \\ \Gamma_{q,k}\times G^{k}_{q} \ar[r, "a_{\Gamma}"] & \Gamma_{q,k}\times\Gamma_{q,k}
		\end{tikzcd}
	\end{center}
	commutes.
	Indeed, this follows because the right action of $G^{k}_{q}$ on $\Fr_{k}(\RB^{q})$ commutes with the left action of $\Gamma_{q}$.
	That $Ba_{\Gamma}$ is an induction then follows from the inductivity of $a_{\Gamma}$ by a diagram chase.

	Finally, we come to identifying $B\Gamma_{q,k}$ with $\Fr_{k}(\gamma)$.  For this, observe that for any $\alpha\in\NB$ we have a canonical, $G_{q}^{k}$-equivariant identification of $U_{\alpha} \subset B\Gamma_{q,k}$ with $(s\circ\gamma_{\alpha\alpha})^{*}\Fr_{k}(\RB^{q})$, defined by
	\begin{equation}\label{locid}
		[\vec{t};\vec{\alpha};\overrightarrow{(\gamma,\phi)}]\mapsto\big([\vec{t};\vec{\alpha};\vec{\gamma}],\,\phi_{j}\big),
	\end{equation}
	where $\alpha = \alpha_{j}\in\vec{\alpha}$.
	Furthermore, if $[\vec{t};\vec{\alpha};\overrightarrow{(\gamma,\phi)}]\in U_{\alpha}\cap U_{\beta}$, where $\beta = \alpha_{k}\in\vec{\alpha}$, then one has $\phi_{j} = \gamma_{\alpha\beta}([\vec{t};\vec{\alpha};\vec{\gamma}])\cdot\phi_{k}$.  It follows that the local identifications of Equation \eqref{locid} patch together to a global identification of $B\Gamma_{q,k}$ with $\Fr_{k}(\gamma)$.
\end{proof}

Recall the natural identification of $G^{1}_{q}$ with $\GL(\RB^{q})$.
Recall too the inclusions
\[ \O(q) \xhookrightarrow{\quad} \GL(\RB^{q}) \xhookrightarrow{\quad} G^{\infty}_{q} , \]
defined by taking the infinite jet at \( 0 \) of diffeomorphisms that fix the metric and linear structures of \( \RB^q \) respectively.
The construction of a section of $\Fr_{\infty}(\gamma)/\O(q)\rightarrow B\Gamma_{q}$ requires the following lemma, which defines a canonical `exponential' path in \( G^{\infty}_{q}/\O(q) \) to each element from the identity equivalence class.  For this we make use of the fact that, as a projective limit of manifolds, $G_{q}^{\infty}$ admits a natural tangent structure \cite[Chapter 1]{varbi}.
\begin{lemma}\label{exp}
	There exists a $G^{1}_{q}$-equivariant diffeomorphism $\jetexp \colon \tBun_{\O(q)}\bigl(G^{\infty}_{q}/\O(q)\bigr)\rightarrow G^{\infty}_{q}/\O(q)$.
\end{lemma}
\begin{proof}
	The natural projection \( \pi \colon G^{\infty}_{q}\rightarrow G^{1}_{q} \) of groups is split, so, letting \( N = \ker(\pi)\), there is an identification of \( G^{\infty}_{q} \) as the semidirect product \( G^{1}_{q}\ltimes N \). Explicitly, the identification is given by
	\begin{equation*}\label{semidirect}
		G^{\infty}_{q}\ni g \xmapsto{\qquad} \bigl(\pi(g),g\pi(g)^{-1}\bigr)\in G^{1}_{q}\ltimes N .
	\end{equation*}
	The right action of \( \O(q) \) only disturbs the first factor, so induces a diffeomorphism $G^{\infty}_{q}/\O(q)\cong G^{1}_{q}/\O(q)\times N$.
	This diffeomorphism is equivariant for the left action of $G^{1}_{q}$, in that
	\[ A [g] \mapsto (A[\pi(g)], A g \pi(g)^{-1}A^{-1}) \]
	for \( A \in G^{1}_{q} \) and \( [g] \in G^{\infty}_{q}/\O(q) \).
	This gives also the identification
	\[ \tBun_{\O(q)}\bigl(G^{\infty}_{q}/\O(q)\bigr) \cong (\gf^{1}_{q}/\so(q)) \oplus \mathfrak{n} . \]

	There is a canonical bi-$\O(q)$-invariant, left-$G^{1}_{q}$-invariant Riemannian structure on $G^{1}_{q}$, which descends to a Riemannian structure on the homogeneous space $G^{1}_{q}/\O(q)$.
	The Riemannian exponential map for this metric induces a diffeomorphism $\exp_{R} \colon \tBun_{\O(q)}(G^{1}_{q}/\O(q))\rightarrow G^{1}_{q}/\O(q)$, which is equivariant for the left multiplication actions of \( G^{1}_{q} \) \cite[Chapter VI, Theorem 1.1]{helgason}.
	From \cite[Proposition 13.4]{natopdiffgeom}, the exponential map $\exp_{N} \colon \mathfrak{n}\rightarrow N$ is a global diffeomorphism, equivariant for the adjoint actions of \( G^{1}_{q} \) on domain and codomain.
	Combining these two facts, we have the diffeomorphism
	\[
	\jetexp = \exp_{R}\times\exp_{N} \colon (\gf^{1}_{q}/\so(q)) \oplus\mathfrak{n} \xrightarrow{\qquad} \bigl(G^{1}_{q}/\O(q)\bigr)\times N ,
	\]
	which is equivariant in each factor, hence equivariant.
\end{proof}

The associated bundle construction works internal to the diffeological category \cite[Section 8.16]{diffeology}.
In particular, given a principal $G^{\infty}_{q}$-bundle $Y\rightarrow X$, the quotient \( Y / \O(q) \) is the associated $G^{\infty}_{q}/\O(q)$-bundle to the action of \( G^{\infty}_{q} \) on \( G^{\infty}_{q}/\O(q) \).
Define its \emph{vertical tangent bundle} \( V(Y/\O(q)) \) as the associated bundle \( Y\times_{G^{\infty}_{q}}T(G^{\infty}_{q}/\O(q)) \).
We have a commutative diagram,
\[ \begin{tikzcd}
	V(Y/\O(q)) \ar[r] \ar[d] & Y/\O(q) \ar[d] \\
	Y / \O(q) \ar[r, "p" swap] & X \ar[l, bend right, "\sigma"]
\end{tikzcd} \]
where the lower three arrows are bundle projections, and the top arrow is given, using the associated bundle construction on both sides, by the rule
\[
\bigl[y,[g,v]\bigr] = \bigl[y\cdot g,[\id,g^{-1}_{*}v]\bigr]\xmapsto{\qquad} \bigl[y\cdot g,[\jetexp(g^{-1}_{*}v)]\bigr]
\]
for \( y \in Y \) and \( [g,v] \in T(G^{\infty}_{q}/\O(q)) \cong G^{\infty}_{q}/\O(q) \times T_{\O(q)}(G^{\infty}_{q}/\O(q)) \).
The top arrow defines fibrewise diffeomorphisms, in that it maps the fibre over any \( y \in Y/\O(q) \) diffeomorphically to the fibre over \( p(y) \).
The next lemma follows immediately.
\begin{lemma}\label{exp2}
	For any section $\sigma \colon X \to Y/\O(q)$, there is a canonical diffeomorphism of fibre bundles
	\[ \jetexp_{\sigma} \colon \sigma^{*}V(Y/\O(q))\rightarrow Y/\O(q) . \]
\end{lemma}

Using the equivalence between \( \Fr_{\infty}(\gamma) \) and \( B\Gamma_{q,\infty} \), we construct a section of $\Fr_{\infty}(\gamma)/\O(q)\rightarrow B\Gamma_{q}$.
This is done semi-simplicially, using Lemma \ref{exp2} in the construction of sections
$\sigma_{k} \colon \Delta_{k}\times\NB^{(k)}\times\Gamma_{q}^{(k)}\rightarrow\Delta_{k}\times\NB^{(k)}\times(\Gamma_{q,\infty}^{(k)}/\O(q))$ that satisfy
\begin{equation}\label{sigmakidentities}
	(\id\times\partial_{j})\circ\sigma_{k}\circ(d_{j}\times\id) = (d_{j}\times\id)\circ\sigma_{k-1}\circ(\id\times\partial_{j}) \qquad j=0,\dots,k ,
\end{equation}
guaranteeing the sections glue to a global section.
We take inspiration from Dupont \cite{dupont2}.

Fix any section $\sigma$ of
\[ \Gamma_{q,\infty}^{(0)} \cong \Fr_{\infty}(\RB^{q}) \xrightarrow{\quad} \RB^{q} \cong \Gamma_{q}^{(0)} , \]
and let $\sigma_{0} \colon \Delta_{0}\times\NB^{(0)}\times\Gamma_{q}^{(0)}\rightarrow \Delta_{0}\times\NB^{(0)}\times(\Gamma_{q,\infty}^{(0)}/\O(q))$ be the induced map
\[
\sigma_{0} \colon (*;\alpha;\vec{x})\xmapsto{\qquad} \bigl(*;\alpha;[\sigma(\vec{x})]\bigr), \qquad\vec{x}\in\RB^{q}.
\]
We now construct $\sigma_{k}$ for $k>0$.
The diffeomorphism $\jetexp \colon \sigma^{*}V(\Fr_{\infty}(\RB^{q})/\O(q))\rightarrow \Fr_{\infty}(\RB^{q})/\O(q)$ of Lemma \ref{exp2} allows the definition of a fibrewise contraction:
for $t\in[0,1]$ and $\vec{x}\in\RB^{q}$, define $g_{t,\vec{x}} \colon \Fr_{\infty}(\RB^{q})/\O(q)_{\vec{x}}\rightarrow \Fr_{\infty}(\RB^{q})/\O(q)_{\vec{x}}$ by the formula
\[
g_{t,\vec{x}}(b):=\jetexp_{\sigma}(t\jetexp^{-1}_{\sigma}(b)),\qquad  b\in \Fr_{\infty}(\RB^{q})/\O(q)_{\vec{x}}.
\]
Now letting $(t_{0},\dots,t_{k})$ denote the barycentric coordinates on the standard simplex $\Delta_{k}$, write $s_{i}:=t_{i}+\cdots+t_{k}$ for $i=1,\dots,k$, and define $\sigma_{k} \colon \Delta_{k}\times\NB^{(k)}\times\Gamma_{q}^{(k)}\rightarrow\Delta_{k}\times\NB^{(k)}\times(\Gamma_{q,\infty}^{(k)}/\O(q))$  by the formula
\[
\sigma_{k}(\vec{t};\vec{\alpha};\vec{\gamma}):=(\vec{t};\vec{\alpha};\tilde{\sigma}^{1}_{k}(\vec{t};\vec{\gamma}),\dots,\tilde{\sigma}^{k}_{k}(\vec{t};\vec{\gamma})) ,
\]
where
\begin{equation}\label{sigmak}
	\tilde{\sigma}^{i}_{k}(\vec{t};\vec{\gamma}):=\gamma_{k}^{-1}\cdots\gamma_{i}^{-1}\cdot g_{s_{1},r(\gamma_{1})}\bigg(\gamma_{1}\cdot g_{\frac{s_{2}}{s_{1}},r(\gamma_{2})}\bigg(\gamma_{2}\cdot\dots\cdot g_{\frac{s_{k}}{s_{k-1}},r(\gamma_{k})}\big(\gamma_{k}\cdot\sigma(s(\gamma_{k}))\big)\cdots\bigg)\bigg) .
\end{equation}
As in \cite[p. 241]{dupont2}, $\sigma_{k}$ may be assumed to be smooth, and a routine calculation shows that they satisfy the identities of Equation \eqref{sigmakidentities}.  As a consequence, the $\sigma_{k}$ glue to a smooth section $B\sigma$ of $\Fr_{\infty}(\gamma)/\O(q)\cong B\Gamma_{q,\infty}/\O(q)\rightarrow B\Gamma_{q}$.

We can now easily prove Theorem \ref{sections}.  According to the above construction, smooth sections of $\Fr_{\infty}(\gamma)/\O(q)\rightarrow B\Gamma_{q}$ exist.  Now let $h$ be a smoothly numerable Haefliger structure on a diffeological space $X$, associated to a smooth map $\eta \colon X\rightarrow B\Gamma_{q}$.  From the isomorphism $\Fr_{\infty}(h)/\O(q)\cong\eta^*\Fr_{\infty}(\gamma)/\O(q)$ (Proposition \ref{functorial}), any smooth section of $\Fr_{\infty}(\gamma)/\O(q)\rightarrow B\Gamma_{q}$ induces a corresponding smooth section of $\Fr_{\infty}(h)/\O(q)\rightarrow X$.  Finally, any two smooth sections are smoothly homotopic via an exponentiated-linear homotopy using Lemma \ref{exp2}.

	\bibliographystyle{amsplain}
	\bibliography{references}

\end{document}